\documentclass[twoside,leqno,10pt]{amsart}
\usepackage{amsfonts}
\usepackage{amsmath}
\usepackage{amscd}
\usepackage{amssymb}
\usepackage{amsthm}
\usepackage{amsrefs}
\usepackage{latexsym}
\usepackage{bbm}
\numberwithin{equation}{section}

\begin{document}
\newtheorem{theorem}{Theorem}
\newtheorem{lemma}{Lemma}
\newtheorem{corollary}{Corollary}
\numberwithin{equation}{section}
\newcommand{\dif}{\mathrm{d}}
\newcommand{\intz}{\mathbb{Z}}
\newcommand{\ratq}{\mathbb{Q}}
\newcommand{\natn}{\mathbb{N}}
\newcommand{\comc}{\mathbb{C}}
\newcommand{\rear}{\mathbb{R}}
\newcommand{\prip}{\mathbb{P}}
\newcommand{\uph}{\mathbb{H}}
\newcommand{\fief}{\mathbb{F}}
\newcommand{\majorarc}{\mathfrak{M}}
\newcommand{\minorarc}{\mathfrak{m}}
\newcommand{\sings}{\mathfrak{S}}

\title{On the Low-lying Zeros of Hasse-Weil $L$-functions for Elliptic Curves}
\date{\today}
\author{Stephan Baier \and Liangyi Zhao}
\maketitle

\begin{abstract}
In this paper, we obtain an unconditional density theorem concerning the low-lying zeros of Hasse-Weil $L$-functions for a family of elliptic curves.  From this together with the Riemann hypothesis for these $L$-functions, we infer the majorant of $27/14$ (which is strictly less than 2) for the average rank of the elliptic curves in the family under consideration.  This upper bound for the average rank enables us to deduce that, under the same assumption, a positive proportion of elliptic curves have algebraic ranks equaling their analytic ranks and finite Tate-Shafarevic group.  Statements of this flavor were known previously \cite{Young} under the additional assumptions of GRH for Dirichlet $L$-functions and symmetric square $L$-functions which are removed in the present paper.
\end{abstract}

\noindent {\bf Mathematics Subject Classification (2000)}: 11M06, 11M26, 11M41, 11F30, 11G05, 11G40, 11L20, 11L40. \newline

\noindent {\bf Keywords}: low-lying zeros, elliptic curve $L$-functions, ranks of elliptic curves

\section{Introduction and Statements of Results}
The philosophy of random matrix theory is that statistics associated to zeros of a family of $L$-functions can be modeled by the statistics of eigenvalues of large random matrices in a suitable linear group.
In the present paper, we consider the statistics of zeros of Hasse-Weil $L$-functions associated to elliptic curves over the rationals near the central point $s = 1$. Moreover, assuming the truth of the generalized Riemann hypothesis (GRH), we draw conclusions about the zeros at the central point itself, which by the conjecture of Birch-Swinnerton-Dyer contain important arithmetical information about the relevant elliptic curves. In particular, we prove results on the average analytic rank for the family of all elliptic curves.  Results of this kind have previously been established by
A. Brumer \cite{Brumer}, D. R. Heath-Brown \cite{HB1}, P. Michel \cite{Michel}, J. H. Silverman \cite{JS3}, M. P. Young \cite{Young} and others.  For a more detailed account of the subject, we refer the readers to the well and clearly written survey article of M. P. Young \cite{Young2}. \newline

Most relevant for us in the present paper are the results in \cite{Brumer, HB1, Young}.  A. Brumer \cite{Brumer} and D. R. Heath-Brown \cite{HB1} obtained majorants of the average analytic rank for the family of all elliptic curves under the assumption of GRH for elliptic curve $L$-functions.  M. P. Young \cite{Young} improved these results under GRH for Dirichlet $L$-functions and elliptic curve $L$-functions.  \newline

We consider elliptic curves $E=E_{a,b}$ of the form
\begin{equation} \label{curveform}
 y^2=x^3+ax+b,
\end{equation}
where $a$ and $b$ are integers.  This elliptic curve has discriminant
\[ \Delta = -16 (4a^3+27b^2). \]
We define $\lambda_E(p)$ by the formula
\[ \# E(\fief_p) = p+1-\lambda_E(p), \]
where $\# E(\fief_p)$ is the number of points on $E(\fief_p)$.  If $p \nmid \Delta$, then $\lambda_E(p)$ is the trace of the Frobenius morphism of $E$ over $\fief_p$.  It was a result due to H. Hasse that 
\begin{equation} \label{hasseest}
 |\lambda_E(p)| < 2 \sqrt{p}
\end{equation}
for all primes $p$. Moreover, for any prime $p>3$, $\lambda_E(p)$ is given by the following formula.
\begin{equation}
\lambda_E(p) = - \sum_{x \bmod{p}} \left( \frac{x^3+ax+b}{p} \right),
\end{equation}
where here and after $\left( \frac{\cdot}{p} \right)$ is the Legendre symbol.  To define the Hasse-Weil $L$-function associated with $E$, we first need to transform $E$ to an elliptic curve $E'$ which is in global minimal Weierstrass form
\[ E': y^2+a_1xy+a_3y=x^3+a_2x^2+a_4x+a_6. \]
  This is always possible.  For the details of this transformation, see \cite{JS1}.  Let $\Delta'$ be the discriminant of $E'$.  The Hasse-Weil $L$-function associated with $E_{a,b}$ is given by the Euler product
\begin{equation} \label{lfundef}
 L(s,E) = \prod_{p \nmid \Delta'} \left( 1-\lambda_{E'}(p) p^{-s} + p^{1-2s} \right)^{-1} \prod_{p|\Delta'} \left( 1-\lambda_{E'}(p) p^{-s} \right)^{-1},
\end{equation}
where $\lambda_{E'}(p)$ is defined by
\[ \# E'(\fief_p) = p+1-\lambda_{E'}(p). \]
We note the important fact that $\lambda_{E'}(p) = \lambda_{E}(p)$ for $p>3$ and $p$ not dividing the discriminant of the elliptic curve $E$.  The infinite product in \eqref{lfundef} converges absolutely and uniformly for $\Re s > 3/2$ by the virtue of \eqref{hasseest}.  Due to celebrated results of  A. Wiles {\it et al} \cites{taywil, BCDT, wiles}, these $L$-functions defined in \eqref{lfundef} coincide with $L$-functions of weight two primitive cuspidal new forms, hence enabling us to use analytic information of the latter to extract information of the former.  In particular, the $L$-functions in \eqref{lfundef} have analytic continuation to the whole complex plane.  Moreover, they have explicit formulas in the form of (4.25) in \cite{ILS} which appears as \eqref{startpt} in the present paper.  The critical strip of $L(s,E)$ is $1/2 < \Re s < 3/2$ and the central point is $s=1$. \newline

Keeping in line with the notations of Iwaniec-Luo-Sarnak \cite{ILS}, we set
\[ D(E; \phi) = \sum_{\rho_E} \phi \left( \frac{\gamma^*_E}{2 \pi} \log X \right), \]
where the sum runs over the non-trivial zeros of $L(s,E)$,
\[ \gamma^*_E = -i (\rho_E - 1), \]
$\phi$ is entire and when restricted to the real line is an even Schwartz class test function whose Fourier transform is compactly supported when restricted to the real line and $X$ is a parameter at our disposal.  The Fourier transform of $\phi$ is defined as
\[ \hat{\phi} (y) = \int_{-\infty}^{\infty} \phi(x) e(-xy) \dif x, \; \mbox{with} \; e(x) = \exp ( 2 \pi i x). \]
$D(E; \phi)$ should be considered as representing the density of zeros of $L(s,E)$ near the central point $s=1$. \newline

We note that if $\phi$ is a Schwartz class test function on the real line and $\hat{\phi}$ has compact support, then $\phi$ extends to an entire function on $\comc$. \newline

If GRH holds for $L(s,E)$, then $\gamma_E^* = \gamma_E$, the imaginary part of $\rho_E$. \newline

 We are interested in studying the weighted average
\begin{equation} \label{Ddef}
\mathcal{D} ( \mathcal{F}; \phi, w ) = \sum_{E \in \mathcal{F}} D(E; \phi) w(E),
\end{equation}
where $w(E)$ is a smooth, compactly supported function on $\rear$ and $\mathcal{F}$ is a family of elliptic curves to be specified presently. \newline

We shall establish the following result unconditionally.

\begin{theorem} \label{main} Let $\mathcal{F}$ be the family of elliptic curves given by the Weierstrass equations $E_{a,b}\ :\ y^2=x^3+ax+b$ with $a$ and $b$ positive integers. Let
$w\in C_0^{\infty}(\rear^+\times\rear^+)$ $($for us $\rear^+=(0,\infty))$ and set $w_X(E_{a,b})=w\left(\frac{a}{A},\frac{b}{B}\right)$, where $A=X^{1/3}$, $B=X^{1/2}$
$(X$ a positive real number$)$. Then
\begin{equation}
\mathcal{D}\left(\mathcal{F};\phi,w_X\right)\sim \left[\hat{\phi}(0)+\frac{1}{2}\phi(0)\right]W_X(\mathcal{F})\ \ as\ X\rightarrow\infty
\end{equation}
for $\phi$ with support in $(-7/10,7/10)$, where $\mathcal{D} \left(\mathcal{F};\phi,w\right)$ is as defined in \eqref{Ddef} and
\begin{equation} \label{Wdef}
W_X (\mathcal{F}) = \sum_{E \in \mathcal{F}} w_X(E).
\end{equation}
\end{theorem}

Theorem~\ref{main} was established by A. Brumer \cite{Brumer} with $\pm 5/9$ in place of $\pm 7/10$ in our theorem and by D. R. Heath-Brown \cite{HB1} with $\pm 2/3$ in place of $\pm 7/10$.  Moreover, both of the last-mentioned results require GRH for elliptic curve $L$-functions.  M. P. Young \cite{Young} obtained Theorem~\ref{main} with $\pm 7/9$ in place of $\pm 7/10$ under the assumption of GRH for Dirichlet and symmetric square $L$-functions. \newline

We further note here that $7/10>2/3$.  It was noted in \cite{HB1} that having an admissible range of support larger than $(-2/3, 2/3)$ is of ``paramount importance'' as such a larger range would permit an upper bound strictly smaller than 2 for the average analytic rank of elliptic curves.  Indeed, from Theorem~\ref{main}, we have the following

\begin{corollary} \label{cor1}
Assuming GRH for Hasse-Weil $L$-functions, the family of elliptic curves ordered as in Theorem \ref{main} has average rank $r\le 1/2+10/7=2-1/14=27/14$.
\end{corollary}

\begin{proof}
It was shown in \cite{ILS} that the upper bound for the average rank takes the form $1/2+1/\nu$ where $[-\nu, \nu]$ is contained in the support of $\hat{\phi}$ in Theorem~\ref{main}.  GRH for Hasse-Weil $L$-functions enables us to discard all zeros that are not central by positivity of $\phi$ on the real line.
\end{proof}

The respective versions of Theorem~\ref{main} due to Brumer \cite{Brumer} and Heath-Brown \cite{HB1}, with their admissible ranges of support for $\hat{\phi}$, lead to the bound for average rank $r$ of all elliptic curves of $r \leq 23/10$ and $r \leq 2$.  Young's version of Theorem~\ref{main} \cite{Young} leads the upper bound of $r \leq 25/14$.  Many believe and it has been conjectured that this average rank should be $1/2$. \newline

It is particularly note-worthy that the majorant in Corollary~\ref{cor1} is strictly less than 2.  Under the truth of such a majorant for the average rank, a positive proportion of elliptic curves have rank either 0 or 1.  Using the famous theorem, chiefly due to Kolyvagin \cite{Ko} and Gross-Zagier \cite{GZ}, that if the analytic rank of an elliptic curve does not exceed 1, then its algebraic and analytic ranks are the same and its Tate-Shafarevich group is finite.  From all this, the following can be deduced.

\begin{corollary} \label{cor2}
Under the assumption of GRH for Hasse-Weil $L$-functions, a positive proportion of elliptic curves ordered as in Theorem \ref{main} have algebraic ranks equal to analytic ranks and finite Tate-Shafarevic groups.
\end{corollary}

In \cite{Young}, this result was proved under the additional assumptions of GRH for Dirichlet and symmetric square $L$-functions. The main outcome of this paper is that we have removed these assumptions. Making the assertion in Corollary \ref{cor2} completely unconditionally seems to be out of reach of the currently available methods. It even seems to be extremely difficult to prove an unconditional bound for the average analytic rank of all elliptic curves (see section 12).

\section{General Approach}
Our starting point is (4.25) of \cite{ILS} which is the explicit formula for $L(s,E)$.  The formula is
\begin{equation} \label{startpt}
D(E; \phi) = \hat{\phi}(0) \frac{\log N_E}{\log X} + \frac{1}{2} \phi(0) - P_1(E; \phi) - P_2(E; \phi) + O \left( \frac{\log \log |\Delta|}{\log X} \right),
\end{equation}
where $N$ is the conductor of the elliptic curve $E$,
\[ P_1(E;\phi) = \sum_{p >3} \lambda_{E'}(p) \hat{\phi} \left( \frac{\log p}{\log X} \right) \frac{2 \log p}{p \log X} \; \mbox{and} \; P_2(E;\phi) = \sum_{p>3} \lambda_{E'}(p^2) \hat{\phi} \left( \frac{2 \log p}{\log X} \right) \frac{2 \log p}{p^2 \log X}, \]
and $X$ is some scaling parameter. We note that if $p \nmid \Delta$ and $p>3$, then
\begin{equation} \label{lambdaform}
\lambda_{E'}(p) = \lambda_{E}(p) = - \sum_{x \bmod{p}} \left( \frac{x^3+ax+b}{p} \right)
=: \lambda_{a,b}(p)
\end{equation}
and
$$
\lambda_{E'}(p^2) = \lambda_{a,b}(p)^2-p,
$$
provided that $E$ is of the form in \eqref{curveform}. We set
\[ \tilde{P_1}(E;\phi) = \sum_{p >3} \lambda_{a,b}(p) \hat{\phi} \left( \frac{\log p}{\log X} \right) \frac{2 \log p}{p \log X} \; \mbox{and} \; \tilde{P_2}(E;\phi) = \sum_{p>3} \lambda_{a,b}(p^2) \hat{\phi} \left( \frac{2 \log p}{\log X} \right) \frac{2 \log p}{p^2 \log X}, \]
where $\lambda_{a,b}(p)$ is defined as in \eqref{lambdaform} and $\lambda_{a,b}(p^2)$ is defined by
\begin{equation} \label{atsquares}
\lambda_{a,b}(p^2) := \lambda_{a,b}(p)^2-p.
\end{equation}
We point out that if $p>3$ and $p|\Delta$, then $\lambda_{a,b}(p^2)$ defined above differs from $\lambda_E(p^2)$ since in this case we have $\lambda_E(p^2)=\lambda_E(p)^2$. However, if $p>3$ and $p\nmid \Delta$, then $\lambda_{a,b}(p^2)=\lambda_E(p^2)=\lambda_{E'}(p^2)$. \newline

By Hasse's bound $|\lambda_{a,b}(p)|\le 2\sqrt{p}$ and the elementary facts that
\begin{equation} \label{elefact}
\sum_{p|\Delta} \frac{\log p}{p} \ll \log \log \Delta  \quad \mbox{and} \quad \sum_{p|\Delta} \frac{\log p}{\sqrt{p}} \ll \sqrt{\log \Delta \cdot \log \log \Delta},
\end{equation}
we conclude that
$$
\tilde{P_1}(E;\phi)=P_1(E;\phi)+O\left(\frac{\sqrt{\log \Delta \cdot \log \log \Delta}}{\log X}\right)
$$
and
$$
\tilde{P_2}(E;\phi)=P_2(E;\phi)+O\left(\frac{\log \log \Delta}{\log X}\right).
$$
Hence, the explicit formula \eqref{startpt} can be written the form
\begin{equation} \label{startpt1}
D(E; \phi) = \hat{\phi}(0) \frac{\log N_E}{\log X} + \frac{1}{2} \phi(0) - \tilde{P_1}(E; \phi) - \tilde{P_2}(E; \phi) + O \left((\log X)^{-1/2+\varepsilon}\right)
\end{equation}
since we will have $\log \Delta \ll \log X$.
\newline

Now summing \eqref{startpt1} over the family, $\mathcal{F}$, of elliptic curves, we have
\[ \mathcal{D} ( \mathcal{F}; \phi, w_X) = \hat{\phi}(0) \sum_{E \in \mathcal{F}} \frac{\log N_E}{\log X} w_X(E) + \frac{1}{2} \phi(0) W_X(\mathcal{F}) - \mathcal{P}_1 (\mathcal{F}; \phi, w_X) - \mathcal{P}_2 (\mathcal{F}; \phi,w_X) + O \left(\frac{W_X(\mathcal{F})}{(\log X)^{1/2-\varepsilon}} \right), \]
where for $i \in \{ 1,2 \}$
\[ \mathcal{P}_i (\mathcal{F}; \phi, w_X) = \sum_{E \in \mathcal{F}} \tilde{P_i}(E;\phi) w_X(E), \quad \mbox{and} \quad  W_X(\mathcal{F}) =\sum_{E \in \mathcal{F}} w_X(E). \]

Now to prove Theorem~\ref{main}, we need to show that
\begin{equation} \label{boundsforPs}
 \mathcal{P}_i (\mathcal{F}; \phi, w_X) = o(AB) \quad \mbox{as} \quad X\rightarrow \infty
\end{equation}
holds for both $i=1$ and $i=2$ with $A$, $B$ and $\mathcal{F}$ as in Theorem~\ref{main}.  Moreover, a conductor condition needs to be proved, namely
\begin{equation} \label{conductorcond}
\sum_{E \in \mathcal{F}} \frac{\log N_E}{\log X} w_X(E) \sim W_X( \mathcal{F} ) \; \mbox{as} \; X \to \infty.
\end{equation}
We shall establish \eqref{boundsforPs} with $i=2$ by Lemma~\ref{sym2rh} in section 5.  The conductor condition \eqref{conductorcond} for the $A$, $B$, $\mathcal{F}$ and $w$ most relevant for us was already established in \cite{Young} and we simply quote the result here.

\begin{lemma}
Let $A$, $B$, $\mathcal{F}$ and $w$ be as in Theorem~\ref{main}.  Then we have
\[ \sum_{E \in \mathcal{F}} \frac{\log N_E}{\log X} w_X(E) = W_X( \mathcal{F} ) \left( 1+ O \left( \frac{1}{\log X} \right) \right). \]
\end{lemma}

\begin{proof} This is Lemma 5.1 in \cite{Young}. \end{proof}

In the remainder of this section, we shall chiefly focus on the transformation and preparation of $\mathcal{P}_1 (\mathcal{F}; \phi, w_X)$, the unconditional estimation of which will be central to Theorem~\ref{main}. \newline

Following the computations in section 5.2 of \cite{Young}, using Lemma~\ref{poisuml}, \eqref{lambdaform} and Lemma 5.6 in \cite{Young}, we have that
\begin{equation} \label{P1trans1}
 \mathcal{P}_1 (\mathcal{F}; \phi, w_X) = - \frac{AB}{\log X} \sum_{p>3} \psi_4(p) \frac{2 \log p}{p^{3/2}} \hat{\phi} \left( \frac{\log p}{\log X} \right) \sum_h \sum_k \left( \frac{k}{p} \right) e \left( - \frac{h^3 \bar{k}^2}{p} \right) \hat{w} \left( \frac{hA}{p}, \frac{kB}{p} \right),
\end{equation}
where $\psi_4(p)$ is the sign of the quadratic Gauss sum
\[ \sum_{\beta \bmod{p}} \left( \frac{\beta}{p} \right) e \left( \frac{\beta}{p} \right) \]
and is 1 if $p \equiv 1 \pmod{4}$ and $i$ if $p \equiv -1 \pmod{4}$.  We know that $\psi_4(p)$ is a Dirichlet character modulo 4. \newline

The contribution of $h=0$ in \eqref{P1trans1} will be negligible by trivial considerations.  Now breaking the summation ranges on the right-hand side of \eqref{P1trans1} into dyadic segments using a smooth partition of unity, it suffices to estimate
\begin{equation} \label{Sdef}
S(H,K,P) = \sum_h \sum_k \sum_p \frac{\log p}{p^{3/2}} \psi_4(p) \left( \frac{k}{p} \right) e \left( - \frac{h^3 \bar{k}^2}{p} \right) \hat{w} \left( \frac{hA}{p}, \frac{kB}{p} \right) g \left( \frac{h}{H} \right) g \left( \frac{k}{K} \right) g \left( \frac{p}{P} \right),
\end{equation}
where $g$ is a smooth compactly supported function arising from the partition of unity and the support of $g$ is in the interval $[1,2]$.  As noted in \cite{Young}, it suffices to show that
\[ S(H,K,P) \ll X^{-\varepsilon}, \]
with the implied constants depending on $w$, $\| g \|_{\infty}$ and $\| g' \|_{\infty}$.  Moreover, it is enough to restrict $H$, $K$ and $P$ to
\[ H\ll (P/A)^{1+\varepsilon}, \; K\ll (P/B)^{1+\varepsilon}, \; \mbox{and} \; P \ll X^{7/10-\varepsilon}. \]
If the above restrictions are not met, then the corresponding contributions are negligible due to the rapid decay of the Fourier transform $\hat{w}$.  Using Weyl's method for exponential sums, we get
\[ \sum_{P \leq p <2P} \left| \sum_{H \leq h <2H} e \left( \frac{h^3 \bar{k}^2}{p} \right) \right| \ll (H^{3/4}P + HP^{3/4}+H^{1/4}P^{5/4})(HKP)^{\varepsilon}. \]
The above estimate applied to \eqref{Sdef} after partial summation will lead to the bound
\[ S(H,K,P) \ll (H^{3/4}KP^{-1/2} + HKP^{-3/4}+H^{1/4}P^{-1/4}) X^{\varepsilon}, \]
which is (15) in \cite{Young}.  Using this estimate and $H\ll (P/A)^{1+\varepsilon}$ and $P\ll X$ it follows that
\[ S(H,K,P)\ll X^{-\varepsilon} \]
if $K\ll \min\left\{P^{1/2}H^{-3/4},P^{1/4}H^{-1/4}\right\}X^{-2\varepsilon}$. Therefore, we may assume from now on that
\begin{equation} \label{small}
K\gg \min\left\{P^{1/2}H^{-3/4},P^{1/4}H^{-1/4}\right\}X^{-2\varepsilon}.
\end{equation}
We now transform $S(H,K,P)$ {\it \'a la} Young \cite{Young}.  We have
\begin{equation} \label{transS1}
S(H,K,P) = \sum_{K \leq k \leq 2K} \sum_{d|k^2} \frac{1}{\varphi(k^2/d)} \sum_{\chi \bmod{k^2/d}} \tau(\chi) \bar{\chi} (d_0^3/d) Q(d,k,\chi),
\end{equation}
where $\tau(\chi) = \tau_1(\chi)$, the usual Gauss sum, as defined in \eqref{gausssumdef},
\[ Q(d,k,\chi) = \sum_{P \leq p < 2P} \sum_{H/d_0 \leq h <2H/d_0} \psi_4(p) \chi(p) \left( \frac{k}{p} \right) \bar{\chi}^3 (h) e \left( - \frac{h^3d_0^3}{pk^2} \right) U(h,d_0,k,p), \]
\[ U(h,d_0,k,p) = g \left( \frac{hd_0}{H} \right) g \left( \frac{k}{K} \right) g \left( \frac{p}{P} \right) \hat{w} \left( \frac{hd_0A}{p}, \frac{kB}{p} \right) \frac{2 \log p}{p^{3/2}\log X} \hat{\phi} \left( \frac{\log p}{\log X} \right), \]
and $d_0$ is the least positive integer such that $d|d_0^3$.  Much of what remains in the paper will be devoted to the estimation of $Q(d,k,\chi)$. \newline

\section{Notations}
As usual, by $\varepsilon$ we denote an arbitrarily small positive number which may change from line to line.  If $\chi$ is a Dirichlet character and $\rho$ is a zero of the Dirichlet $L$-function $L(s, \chi)$, then we denote its real part by $\beta$ and its imaginary part by $\gamma$.  Moreover, if $T>0$ and $\sigma\ge 0$, then we denote the number of zeros $\rho$ of  $L(s, \chi)$ with $\beta\ge \sigma$ and $|\gamma|\le T$ by $N(\chi,T,\sigma)$. Furthermore, we denote the set of zeros $\rho$ of $L(s, \chi)$ with $\beta\ge \sigma$ by ${\mathcal N}(\chi,\sigma)$.  If $d$ is a natural number, then we define $d'$ to be the smallest natural number such that $d|{d'}^2$, $d^*$ to be the square-free kernel of $d$ (that is, $d^* = {d'}^2/d$), and $d_0$ to be the smallest natural number such that $d|d_0^3$. The symbol $p$ is reserved for prime numbers.\newline

\section{Preliminaries}

First, we shall need the following estimate from \cite{Young}.

\begin{lemma} \label{Hest}
Let $F(u,v)$ be a smooth function satisfying
\begin{equation} \label{fcond}
F^{(\alpha_1,\alpha_2)}(u, v)u^{\alpha_1}v^{\alpha_2}\le C(\alpha_1,\alpha_2)\left(
1+\frac{|u|}{U}\right)^{-2}\left(1+\frac{|v|}{V}\right)^{-2}
\end{equation}
for any $\alpha_1,\alpha_2\ge 0$, the superscript on $F$ denoting partial differentiation.
For $s_i=\sigma_i+it_i$, $i=1,2$ with $1/2 \le \sigma_i\le 1$ define
\begin{equation} \label{Hdef}
H(s_1,s_2)=\int\limits_0^{\infty}\int\limits_0^{\infty} e\left(\frac{u^c}{vq}\right)u^{s_1}v^{s_2}F(u,v)\frac{{\rm d}u{\rm d}v}{uv},
\end{equation}
where $c$ is a positive integer.
Then we have
\begin{equation} \label{H2}
H(s_1,s_2)\ll
\frac{U^{\sigma_1}V^{\sigma_2}}{(1+|t_1/c+t_2|)^2(1+|t_1|)^{1+\varepsilon}}\left(1+\frac{U^c}{V|q|}\right)^{1/2},
\end{equation}
where the implied $\ll$-constant depends only on $\varepsilon$ and $c$.
\end{lemma}

\begin{proof} In Appendix A in \cite{Young} it is proved that
\begin{equation} \label{H1}
H(s_1,s_2)\ll
\frac{U^{\sigma_1}V^{\sigma_2}}{|(s_1/c+s_2)(s_1/c+s_2+1)|\cdot |s_1|^{1+\varepsilon}}\left(1+\frac{U^c}{V|q|}\right)^{1/2}.
\end{equation}
We note that this proof is valid for all $\sigma_i$ with $1/2\le \sigma_i\le 1$ ($i=1,2$). We further note that the term $|s_2|^{1+\varepsilon}$ in the first estimate in Appendix A in \cite{Young} is a mis-print and needs to be replaced by $|s_1|^{1+\varepsilon}$. From (\ref{H1}) one easily deduces (\ref{H2}).
\end{proof}

We shall further use the following lemma on the number of zeros of Dirichlet $L$-functions.

\begin{lemma} \label{zeronumb}
Let $T$ be any real number and $\chi$ be a Dirichlet character with conductor $l$. Then the number of zeros of $L(s,\chi)$ with imaginary part lying in an interval of the form $[T,T+1]$ is bounded by
\[ \ll\log(l(2+|T|)). \]
\end{lemma}

\begin{proof} This follows immediately from Lemma 2.6.4. in \cite{Bru}.
\end{proof}

We further use the following subconvexity bound for Dirichlet $L$-functions in the conductor aspect which is a consequence of the bound for character sums due to D. A. Burgess \cite{DAB}.

\begin{lemma} \label{Lest} Let $\chi$ be a Dirichlet character modulo $l$ with conductor $q$. Then
\begin{equation}
L\left(\frac{1}{2}+it,\chi\right)\ll (1+|t|)q^{3/16}l^{\varepsilon},
\end{equation}
where the implied $\ll$-constant depends only on $\varepsilon$.
\end{lemma}

\begin{proof}
This follows from Theorem 12.9. in \cite{IwKo}.
\end{proof}

In fact, a much weaker bound than the one given in Lemma~\ref{Lest} would suffice for our purposes.  Next, we shall also use the following bound for the logarithmic derivative of $L(s,\chi)$.

\begin{lemma} \label{logdiv}
Let $s_0=\sigma_0+it_0$ be any complex number with $1/2\le \sigma_0\le 1$ and $\chi$ be a Dirichlet character with conductor $l$. Assume that $s_0$ has distance at least $D$ to all zeros of $L(s,\chi)$. Then
\[ \frac{L'}{L}(s_0,\chi)\ll \log(l(2+|t_0|))\left(1+\frac{1}{D}\right). \]
\end{lemma}

\begin{proof} This is an immediate consequence of Lemma 2.6.4. and Satz 2.6.2. in \cite{Bru}.
\end{proof}

We shall also need the following result on the moments of Dirichlet $L$-functions.

\begin{lemma} \label{second}
Let $r=2$ or $r=4$. Then, for $T>0$, we have
\[
\sum\limits_{\chi \ \mbox{\rm \scriptsize mod } q}\ \int\limits_{-T}^T |L(1/2+it,\chi)|^{r} {\rm d}t \ll (qT)^{1+\varepsilon},
\]
where the implied $\ll$-constant depends only on $\varepsilon$.
\end{lemma}

\begin{proof} This follows from Theorem 10.1 in \cite{Mon}.
\end{proof}

Furthermore, we shall use the following familiar bounds and evaluations for Gauss sums which, for a Dirichlet character $\chi$ of conductor $l$ and $\gcd(a,l)=1$, are defined to be
\begin{equation} \label{gausssumdef}
 \tau_a(\chi) = \sum_{b\bmod{l}} \chi(b) e \left( \frac{ab}{l} \right).
\end{equation}

\begin{lemma} \label{gauss}
Let $a$ and $l$ be natural numbers with $\gcd(a,l)=1$ and $\chi$ be a Dirichlet character with conductor $l$, then we have the following.
\begin{enumerate}
\item The Gauss sum is estimated in the following way.
\[ |\tau_a(\chi)|\le \sqrt{l}. \]

\item If $\chi$ is real and $l$ is odd and square-free,
\[ \tau_a(\chi) = \left\{ \begin{array}{ll} \chi(a) \sqrt{l} & if \; l \equiv 1 \pmod{4}, \\ i \chi(a) \sqrt{l} & if \; l \equiv -1 \pmod{4}. \\  \end{array} \right. \]

\item For $k \in \intz$, we have
\[ \left| \sum_{b \bmod{l}} e \left( \frac{ab^2+kb}{l} \right) \right| \leq 2 \sqrt{l}. \]
\end{enumerate}
\end{lemma}

\begin{proof} (1) and (2) follow immediately from Lemmas 3.1. and 3.2. in \cite{IwKo}.  (3) is (7.4.2) in \cite{GK}.
\end{proof}

In our estimates and evaluations, we shall use the following identity.

\begin{lemma} [Poisson Summation modulo $l$] \label{poisuml}
Let $w$ be a Schwartz-class function, $D$ a positive real number an $a$ and integer.  Then we have
\[ \sum_{\substack{d \in \intz \\ d \equiv a \bmod{l}}} w \left( \frac{d}{D} \right) = \frac{D}{l} \sum_{h \in \intz} \hat{w} \left( \frac{hD}{l} \right) e \left( \frac{ha}{l} \right). \]
\end{lemma}
\begin{proof}  This is (4.24) in \cite{IwKo}. \end{proof}

%
%

For the estimation of certain character sums, we shall also need the following two large sieve estimates.

\begin{lemma} \label{ls} Let $q$, $N$ and $M$ be natural numbers. Further, let
$(a_n)_{n\in \mathbbm{N}}$ be a sequence of complex numbers. Then
\[ \sum\limits_{\chi \mbox{\rm \scriptsize\ mod } q} \left\vert \sum\limits_{M\le n<M+N}
a_n\chi(n) \right\vert^2 \leq (q+N) \sum\limits_{M\le n<M+N} |a_n|^2 . \]
\end{lemma}

\begin{proof} This follows easily by expanding the modulus square and
 applying the orthogonality relation of Dirichlet characters. \end{proof}

\begin{lemma} [Heath-Brown] \label{Heathls}
Let $P$ and $N$ be positive integers, and let $a_1,...,a_N$ be arbitrary complex numbers. Then
\[
\sum_{p\sim P} \left| \sum_{n \leq N} a_n \left( \frac{n}{p} \right) \right|^2\ll_{\varepsilon} (PN)^{\varepsilon}(P+N)\sum\limits_{q\le N} \sum\limits_{\substack{n_1,n_2\leq N\\ n_1n_2=q^2}} |a_{n_1}a_{n_2}|
\]
for any $\varepsilon > 0$.
\end{lemma}

\begin{proof} This is an immediate consequence of Corollary 2 in \cite{DRHB}. \end{proof}

We shall require that the number of characters $\chi$ modulo $q$ satisfying an equation of the form $\overline{\chi}^3=\chi_1$ is small. To this end, we use the following bound.

\begin{lemma} \label{charnumb}
Let $\chi_1$ be a character modulo $q$. Then the number of characters $\chi$ modulo $q$ satisfying $\overline{\chi}^3=\chi_1$ is bounded by $q^{\varepsilon}$.
\end{lemma}

\begin{proof} The group of characters modulo $q$ is isomorphic to $(\mathbbm{Z}/q\mathbbm{Z})^*$. Moreover, it is well-known that for fixed $n$ and any $a$ the number of solutions $x$ of the  congruence $x^n\equiv a$ mod $q$ is bounded by $\ll q^{\varepsilon}$. Now, the desired estimate follows.
\end{proof}

Information about the conductors of characters of order 3 is provided by the following

\begin{lemma} \label{ord3}
Let $\chi$ be a primitive Dirichlet character of order 3 and modulus $q$, then
\[
 q = 9^{a} q^{*},
\]
where $a \in \{ 0, 1\}$, $\gcd (q^*, 6) =1$ and $q^*$ is a square-free
 natural number whose prime divisors are all congruent to 1 modulo 3.
\end{lemma}

\begin{proof}
We first note that the group of Dirichlet characters modulo $p^k$ for
 any prime $p \equiv 2 \pmod{3}$ and $k \in \natn$ is always of order
 $\varphi(p^k)=p^{k-1}(p-1)$ which is not divisible by 3.  Therefore,
 there exists no Dirichlet character of order 3 to these
 moduli. \newline

If $p \equiv 1 \pmod{3}$, then it is well-known that $\left( \intz / p^k
 \intz \right)^{*}$ is cyclic for any $k \in \natn$.  It is
 also well-known that if $g$ is a generator of this cyclic group and $m = \varphi(p^k) = p^{k-1}(p-1)$, then every character defined on $\left( \intz / p^k \intz \right)^{*}$ is of the form
\[
 \chi_a (g^l) = e \left( \frac{al}{m} \right),
\]
with a fixed $a$ modulo $m$.  See, for example, Chapter 3 of
 \cite{IwKo}.  Therefore, $\chi_a$ is cubic if and only if
\[
 a \equiv \pm \frac{m}{3} \pmod{m}.
\]
Therefore,  there are two primitive characters of order 3 modulo $p$ and
 only two characters of order 3 modulo $p^k$ for $k >1$.  Hence the said
 characters modulo $p^k$ with $k>1$ are not primitive as there should be
 at least two characters of order 3 induced by those of modulus
 $p$. \newline

We can also easily observe that there are two primitive characters of
 order 3 modulo 9 and by similar arguments as above, there are no
 primitive characters of order 3 modulo $3^k$ for any $k > 2$. \newline

Therefore, there exist primitive characters of order 3 modulo 9 and any prime $p$
 congruent to 1 modulo 3.  Any other primitive character of order 3 is
 a product of these.  Thus our desired result follows.
\end{proof}

We shall also need the following bound for an average of a special exponential sum.

\begin{lemma} \label{expsum} Let $c_n$ be complex numbers satisfying $c_n\le 1$ and let
\[ R(N,P,d_0,k)=\sum\limits_{P\le p <2P} \left\vert\sum\limits_{N\le n<2N} e\left(\frac{n^3d_0^3}{pk^2}\right)c_n\right\vert. \]
Then
\[ R(N,P,d_0,k)=N^{1/2}P+N^{1/4}P^{5/4}k^{1/2}d_0^{-3/4}. \]
\end{lemma}

\begin{proof} This is quoted from \cite{Young} and is Lemma 5.4. there.
\end{proof}

To estimate certain mean values of Dirichlet polynomials twisted with
characters, we shall use the following two lemmas attributed to Gallagher.

\begin{lemma} [Gallagher] \label{gal1}
Let $T_0$ and $T \geq \delta >0$ be real numbers, $S$ be a continuous complex valued function
 on $[T_0, T_0+T]$ with continuous derivatives on $(T_0, T_0+T)$, and
 $\mathcal{S}$ be a finite set in the interval $[T_0+\delta/2,
 T_0+T-\delta/2]$ with the property that
\[ |t_1-t_2| \geq \delta, \; \mbox{for all distinct} \; t_1, t_2 \in \mathcal{S} .  \]
Then we have
\[ \sum_{t \in \mathcal{S}} |S(t)|^2 \leq \frac{1}{\delta} \int_{T_0}^{T_0+T} |S(t)|^2 \dif t + \left( \int_{T_0}^{T_0+T} |S(t)|^2
				      \dif t \right)^{1/2} \left( \int_{T_0}^{T_0+T} |S'(t)|^2
				      \dif t \right)^{1/2} . \]
\end{lemma}

\begin{proof} This follows from Lemma 1.4 in \cite{Mon} and has its origin in \cite{PXG}.
\end{proof}

\begin{lemma} [Gallagher] \label{gal2}
If
\[
 \sum_{n=1}^{\infty} a_n n^{-s}
\]
is absolutely convergent for $\Re s \geq 0$, then
\[
 \int_{-T}^T \left| \sum_{n=1}^{\infty} a_n n^{-it} \right|^2 \dif t \ll
 T^2 \int_0^{\infty} \left| \sum_y^{\tau y} a_n \right|^2 \frac{\dif y}{y},
\]
where $\tau = \exp (1/T)$ and $T>0$.
\end{lemma}

\begin{proof} This is Lemma 1.10 in \cite{Mon} and has its origin in \cite{Galla}.
\end{proof}

Finally, we shall need the following technical estimates.

\begin{lemma} \label{f}
Let $D\ge 1$ and $a$ be a natural number. Then we have
\begin{equation} \label{f2}
\begin{split}
\sum\limits_{d\le D}
\frac{d{d^*}^{1/2}}{d_0^{1/2}{d'}^{3/2}} \ll D^{1/2+\varepsilon},\ \ \ \ \ &
\sum\limits_{d\le D}
\frac{d{d^*}^{1/2}}{d_0{d'}^{3/2}} \ll D^{1/4+\varepsilon}, \ \ \ \ \ \sum\limits_{d\le D} \frac{d}{(d^4)_0} \ll D^{\varepsilon},\ \ \ \ \ \\
\sum\limits_{d\le D} \frac{1}{{d'}^{1/2}}\ll D^{1/2+\varepsilon}, \ \ \ \ \ &
\sum\limits_{d\le D} \frac{d^{1/3}}{d_0^{2/3}d'}\ll D^{\varepsilon},\ \ \ \ \ \sum\limits_{d\le D} \frac{d^{1/3}{d^*}}{d_0^{2/3}{d'}^3}\ll D^{\varepsilon}.
\end{split}
\end{equation}
\end{lemma}

\begin{proof} To prove the first estimate in \eqref{f2},
we use the fact that the function
\[ f(d)=\frac{d{d^*}^{1/2}}{d_0^{1/2}{d'}^{3/2}} \]
is multiplicative, and, if $p$ is a prime, then
\[ f(p)=p^{-1/2} \ \mbox{ and }\ f(p^r)\le 1 \mbox{ if } r\ge 2. \]
It follows that
\begin{eqnarray} \label{mul}
\sum\limits_{d\le D} f(d)\le D^{1/2} \sum\limits_{d\le D} \frac{f(d)}{d^{1/2}} &\le&
D^{1/2}\prod_{p\le D}\left(\sum\limits_{r=0}^\infty f(p^r)p^{-r/2}\right) \le
D^{1/2}\prod_{p\le D}\left(1+\frac{C}{p}\right)\\ &\le&
D^{1/2}\prod_{p\le D}\left(1-\frac{1}{p}\right)^{-C}\nonumber
\end{eqnarray}
for some constant $C\ge 1$.
It is well-known that the product in the last line of \eqref{mul} is dominated by $\ll D^{\varepsilon}$ which completes the proof of the first estimate in \eqref{f2}. All other estimates in \eqref{f2} can be proved by a similar technique.
\end{proof}

\section{Removing RH for Symmetric Square $L$-functions}

The Riemann hypothesis for symmetric square $L$-functions enabled M. P. Young \cite{Young} to readily dispose of the contribution ${P}_2$ in \eqref{startpt} and infer
\begin{equation} \label{formD}
D ( E; \phi ) = \hat{\phi} (0) \frac{\log N}{\log X} + \frac{1}{2} \phi(0) - P_1 (E; \phi) + O \left( \frac{ \log \log |\Delta|}{\log X} \right)
\end{equation}
which is (7) in \cite{Young}.  In practice, what is actually needed is \eqref{startpt1} averaged over the family of elliptic curves under consideration.  Keeping this in mind, the dependency on RH for symmetric square $L$-functions can be removed as noted in \cite{Young}.  We give this computation here and dispose of the relevant contribution 
$$
\mathcal{P}_2( \mathcal{F}; \phi, w_X)=\sum_{E \in \mathcal{F}} \tilde{P_2}(E;\phi) w_X(E)
$$
unconditionally. 

\begin{lemma} \label{sym2rh}
Set $A=X^{1/3}$ and $B=X^{1/2}$.  We have
\begin{equation} \label{sym2rheq}
\mathcal{P}_2( \mathcal{F}; \phi, w_X) = \sum_{p>3} \sum_a \sum_b \lambda_{a,b} (p^2) \hat{\phi} \left( \frac{2 \log p}{\log X} \right) \frac{ 2 \log p}{p^2 \log X} w \left( \frac{a}{A}, \frac{b}{B} \right) \ll \frac{X^{5/6}}{\log X} \log \log X,
\end{equation}
provided that the support of $\hat{\phi}$ is in $(-1, 1)$.
\end{lemma}

\begin{proof}
We recall that, by definition in \eqref{atsquares},
\[ \lambda_{a,b} (p^2) = \lambda_{a,b}(p)^2 -p. \]
Hence we get
\begin{equation} \label{sym2rh1}
\sum_a \sum_b \lambda_{a,b} (p^2) w\left( \frac{a}{A}, \frac{b}{B} \right) = \sum_a \sum_b \lambda_{a,b} (p)^2 w\left( \frac{a}{A}, \frac{b}{B} \right) - p \sum_a \sum_b w\left( \frac{a}{A}, \frac{b}{B} \right) .
\end{equation}

Applying Poisson summation modulo $p$, Lemma~\ref{poisuml}, the first term on the right-hand side of \eqref{sym2rh1} becomes
\begin{equation} \label{sym2rh2}
 \frac{AB}{p^2} \sum_h \sum_k \mathop{\sum \sum}_{\substack{\alpha \bmod{p} \\ \beta \bmod{p}}} \lambda_{\alpha, \beta} (p)^2 e \left( \frac{\alpha h + \beta k}{p} \right) \hat{w} \left( \frac{hA}{p}, \frac{kB}{p} \right).
\end{equation}

Recall that
\[ \lambda_{a,b} (p) = - \sum_{x \bmod{p}} \left( \frac{x^3+ax+b}{p} \right). \]
By the above identity, the terms corresponding to $h=k=0$ in \eqref{sym2rh2} are
\begin{equation} \label{sym2rh3}
\begin{split}
 \frac{AB}{p^2} & \hat{w} (0,0) \mathop{\sum \sum}_{\substack{\alpha \bmod{p} \\ \beta \bmod{p}}} \left( \sum_{x \bmod{p}} \left( \frac{x^3+\alpha x+\beta}{p} \right) \right)^2  \\
& = AB\hat{w}(0,0) \left( p + \frac{1}{p^2} \mathop{\sum_{x_1} \sum_{x_2}}_{x_1 \neq x_2} \mathop{\sum \sum}_{\substack{\alpha \bmod{p} \\ \beta \bmod{p}}} \left( \frac{x_1^3+\alpha x_1+ \beta}{p} \right) \left( \frac{x_2^3+ \alpha x_2+ \beta}{p} \right) +O(1) \right).
\end{split}
\end{equation}
With the change of variables $\beta \to \beta - x_1^3-\alpha x_1$, the double sum over $\alpha$ and $\beta$ on the right-hand side of \eqref{sym2rh3} becomes
\begin{equation} \label{sym2rh4}
\sum_{\beta \bmod{p}} \left( \frac{\beta}{p} \right) \sum_{\alpha \bmod{p}} \left( \frac{\beta+(x_2^3-x_1^3)+\alpha (x_2-x_1)}{p} \right) .
\end{equation}
Now it is easy to observe if $x_1 \neq x_2$ and $\alpha$ runs over the residue classes modulo $p$, so does
\[ \beta+(x_2^3-x_1^3)+\alpha (x_2-x_1). \]
Therefore, the inner sum of the above is zero by the orthogonality of characters.  From this, we deduce that \eqref{sym2rh2} is
\begin{equation} \label{sym2rh5}
AB \left( \hat{w}(0,0)p +  \frac{1}{p^2} \mathop{\sum_h \sum_k}_{(h,k) \neq (0,0)} \mathop{\sum \sum}_{\substack{\alpha \bmod{p} \\ \beta \bmod{p}}} \lambda_{\alpha, \beta} (p)^2 e \left( \frac{\alpha h + \beta k}{p} \right) \hat{w} \left( \frac{hA}{p}, \frac{kB}{p} \right) \right).
\end{equation}

We are now led to consider
\begin{equation} \label{sym2rh6}
\sum_{\alpha \bmod{p}} \ \sum_{\beta \bmod{p}} \ \left( \sum_{x \bmod{p}} \left( \frac{x^3+\alpha x+\beta}{p} \right) \right)^2 e \left( \frac{\alpha h + \beta k}{p} \right) .
\end{equation}
Due to the presence of
\[ \hat{\phi} \left( \frac{2 \log p}{\log X} \right) \]
on the left-hand side of \eqref{sym2rheq}, it suffices to consider the case
\[ p \ll X^{1/2-\varepsilon}. \]
Moreover, the presence of
\[ \hat{w} \left( \frac{hA}{p}, \frac{kB}{p} \right) \]
in \eqref{sym2rh6} gives that the only contribution that is not negligible is when
\[ k \ll X^{\varepsilon/2} \frac{p}{B} \ll X^{-\varepsilon/2}. \]
Therefore, we only need to consider the contribution of $k=0$.  The corresponding terms in \eqref{sym2rh5} are
\begin{equation} \label{sym2rh7}
\frac{AB}{p^2} \sum_{h\neq 0} \hat{w} \left( \frac{hA}{p}, 0 \right) \left( \mathop{\sum_{x_1} \sum_{x_2}}_{x_1 \neq x_2} \sum_{\alpha} \sum_{\beta} \left( \frac{x_1^3+\alpha x_1+\beta}{p} \right) \left( \frac{x_2^3+\alpha x_2+\beta}{p} \right) e \left( \frac{\alpha h}{p} \right) +O(p^2) \right).
\end{equation}
We now make a number of changes of variables.  First making the change of variables
\[ \beta \to \beta - \alpha x_1 -x_1^3, \]
we get that
\begin{equation*}
\begin{split}
 \mathop{\sum_{x_1} \sum_{x_2}}_{x_1 \neq x_2}& \sum_{\alpha} \sum_{\beta} \left( \frac{x_1^3+\alpha x_1+\beta}{p} \right) \left( \frac{x_2^3+\alpha x_2+\beta}{p} \right) e \left( \frac{\alpha h}{p} \right) \\
& = \mathop{\sum_{x_1} \sum_{x_2}}_{x_1 \neq x_2} \sum_{\beta} \left( \frac{\beta}{p} \right) \sum_{\alpha} \left( \frac{(x_2^3-x_1^3)+\alpha (x_2-x_1) + \beta}{p} \right) e \left( \frac{\alpha h}{p} \right).
\end{split}
\end{equation*}
Now re-writing $\alpha$ as $(\alpha-\beta)\overline{(x_2-x_1)}$, noting that $x_2-x_1$ is prime to $p$ since $x_1 \neq x_2$, the above becomes
\[ \mathop{\sum_{x_1} \sum_{x_2}}_{x_1 \neq x_2} \sum_{\beta} \left( \frac{\beta}{p} \right) e \left( - \frac{\beta h \overline{(x_2-x_1)}}{p} \right) \sum_{\alpha} \left( \frac{(x_2^3-x_1^3)+\alpha }{p} \right) e \left( \frac{\alpha h\overline{(x_2-x_1)}}{p} \right).  \]
Now changing $\beta$ into $\beta(x_2-x_1)$ and $\alpha$ into $\alpha(x_2-x_1)$, this last expression becomes
\[ \mathop{\sum_{x_1} \sum_{x_2}}_{x_1 \neq x_2} \sum_{\beta} \left( \frac{\beta}{p} \right) e \left( - \frac{\beta h}{p} \right) \sum_{\alpha} \left( \frac{(x_2^2+x_1^2+x_1x_2)+\alpha }{p} \right) e \left( \frac{\alpha h}{p} \right).  \]
Finally changing $\alpha$ into $\alpha-(x_2^2+x_1^2+x_1x_2)$, we have
\begin{equation} \label{sym2rh8}
\mathop{\sum_{x_1} \sum_{x_2}}_{x_1 \neq x_2} e \left( - \frac{h(x_2^2+x_1^2+x_1x_2)}{p} \right) \left( \sum_{\beta} \left( \frac{\beta}{p} \right) e \left( - \frac{\beta h}{p} \right) \right)^2.
\end{equation}
The square of the sum in \eqref{sym2rh8} is that of a Gauss sum which, by (2) of Lemma~\ref{gauss}, is exactly $p$ if $p$ is congruent to 1 modulo 4 and $-p$ if $p$ is congruent to $-1$ modulo 4.  We only give the computations for the positive cases in what follows as the other cases differ only by a negative sign.  In these cases, \eqref{sym2rh8} is
\begin{equation} \label{sym2rh9}
\begin{split}
 p \mathop{\sum_{x_1} \sum_{x_2}}_{x_1 \neq x_2} &e \left( - \frac{h(x_2^2+x_1^2+x_1x_2)}{p} \right) \\
 & = p \left( \sum_{x_1} e \left( - \frac{hx_1^2}{p} \right) \sum_{x_2} e \left( - \frac{h(x_2^2+x_1x_2)}{p} \right)- \sum_x e \left( - \frac{3hx^2}{p} \right)  \right)
\end{split}
\end{equation}
On the right-hand side of the above, the sums over $x$ and $x_2$ are both quadratic Gauss sums which can again be evaluated.  But we only need to know that they are both of modulus not exceeding $2\sqrt{p}$ by the virtue of (3) of Lemma~\ref{gauss}.  Then summing trivially over $x_1$, we get that \eqref{sym2rh9} is
\[ \ll p^{5/2}. \]
Now, note that the first term in \eqref{sym2rh5} and the second term on the right-hand side of \eqref{sym2rh1} differ only by a negligible amount.  Combining everything, summing over $h$ and $p$ trivially and noting that the relevant range for $h$ is $h \ll p^{1+\varepsilon}/A$, we arrive at our desired result.
\end{proof}

\section{Transformation of $Q(d,k,\chi)$}

In the rest of the paper, it remains to show that
\[ S(H,K,P) \ll X^{-\varepsilon} \]
if
\[ H \ll (P/A)^{1+\varepsilon}, \; K \ll (P/B)^{1+\varepsilon} \; \mbox{and} \; P \ll X^{7/10-\varepsilon}. \]
To this end, we first estimate the expression $Q(d,k,\chi)$ in \eqref{transS1} in terms of moments and zeros of
Dirichlet $L$-functions.  To do so, we first transform general sums of the form
\begin{equation} \label{general}
G=\sum\limits_{u\sim U}\sum\limits_{v\sim V} \chi(u)\psi(v)\Lambda(v)e\left(\frac{u^c}{vq}\right)F(u,v),
\end{equation}
where $\chi$ and $\psi$ are Dirichlet characters to the moduli $l_1$ and $l_2$, respectively,
$\Lambda$ is the von Mangoldt function, $q$ is a non-zero rational number, $c$ is a positive integer, and $F(u,v)$ is a smooth function satisfying
\[ F^{\alpha_1,\alpha_2}(u,v)u^{\alpha_1}v^{\alpha_2}\ll C(\alpha_1,\alpha_2)\left(1+\frac{|u|}{U}\right)^{-2}\left(1+\frac{|v|}{V}\right)^{-2} \]
for any $\alpha_1,\alpha_2\ge 0$ with the superscript on $F$ denoting partial differentiation. \newline

Let ${\mathcal P}$ be an infinite path in the complex plane satisfying the following properties:\medskip

(I) ${\mathcal P}$ is a ``zigzag'' path consisting of vertical and horizontal line segments lying in the strip $1/2+\varepsilon\le \Re s \le 1/2+2\varepsilon$. More precisely, there exists an unbounded sequence $(t_n)_{n\in \mathbbm{Z}}$ with $...<t_{-2}<t_{-1}<t_{0}<t_{1}<t_2<...$ and a sequence
$(\sigma_{n})_{n\in \mathbbm{Z}}$ with $1/2+\varepsilon \le \sigma_{n}\le 1/2+2\varepsilon$ and
$\sigma_{n}\not=\sigma_{n+1}$ such that the vertical line segments have end-points $\sigma_n+it_n$ and $\sigma_n+it_{n+1}$, and the horizontal line segments have end-points $\sigma_n+it_{n+1}$ and $\sigma_{n+1}+it_{n+1}$.
\medskip

(II) Any zero $\rho$ of $L(s,\psi)$ has distance $\gg 1/\log(l_2(2+|\Im \rho|))$ to ${\mathcal P}$.\medskip

(III) The part of ${\mathcal P}$ lying in the rectangle with vertices $1/2+\varepsilon-iT$,
$1/2+2\varepsilon-iT$, $1/2+2\varepsilon+iT$ and $1/2+\varepsilon+iT$ has length $\ll T$.\medskip\\
The existence of a path ${\mathcal{P}}$ satisfying the above properties (I), (II) and (III) follows from Lemma \ref{zeronumb}.  Moreover, by Lemma \ref{logdiv} we have that
\[ \frac{L'}{L}(s,\psi)\ll \log^2(2+|\Im s|)l_2^{\varepsilon} \]
for any $s$ lying on ${\mathcal{P}}$. \newline

Now, if both $\chi$ and $\psi$ are non-principal, then, using Mellin transform and the residue theorem from complex analysis, (\ref{general}) can be written in the form
\begin{equation} \label{T1}
G=G_1+G_2,
\end{equation}
with
$$
G_1:=-\frac{1}{(2\pi i)^2} \int\limits_{\Re s_1=1/2} \int\limits_{{\mathcal{P}}} L(s_1,\chi)\frac{L'}{L}(s_2,\psi)H(s_1,s_2){\rm d} s_2{\rm d} s_1 \; \mbox{and} \;
G_2:=\frac{1}{2\pi i}\int\limits_{\Re s=1/2} L(s,\chi) \sum\limits_{\rho\in {\mathcal{N}}_{\mathcal{P}}(\psi)} H(s,\rho) {\rm d} s,
$$
where
${\mathcal{N}}_{\mathcal{P}}(\psi)$ denotes the set of zeros of $L(s,\psi)$ to the right of the path ${\mathcal P}$, and $H(s_1,s_2)$ is defined as in \eqref{Hdef}. If $\psi$ is principal and $\chi$ is not, we have an extra contribution of
\begin{equation} \label{EC}
G_3:=\frac{1}{2\pi i}\int\limits_{\Re s=1/2} L(s,\chi)H(s,1){\rm d}s
\end{equation}
coming from the pole of $L'/L(s,\psi)$ at $s=1$.\newline

In the following, we estimate the expressions $G_1,G_2,G_3$.
Using Lemma \ref{Hest} and Lemma \ref{Lest}, we obtain
\begin{equation}
G_3 \ll \tilde{l}_1^{3/16}U^{1/2}V\left(1+\frac{U^c}{V|q|}\right)^{1/2}(l_1U)^{\varepsilon},
\end{equation}
where $\tilde{l}_1$ is the conductor of $\chi$.\newline

Taking the properties of ${\mathcal P}$ into account and using Lemma \ref{Hest}, the Cauchy-Schwarz inequality and integration by parts, we obtain
\[ G_1 \ll U^{1/2}V^{1/2}\left(1+\frac{U^c}{V|q|}\right)^{1/2}(l_2V)^{\varepsilon}
\left(\int\limits_{0}^{\infty} \frac{1}{(1+T)^{2+\varepsilon}} \int\limits_{-T}^T |L(1/2+it,\chi)|^2 {\rm d} t {\rm d} T
\right)^{1/2}. \]

Using partial summation, we obtain
\begin{equation} \label{T0}
G_2\ll U^{1/2}\left(1+\frac{U^c}{V|q|}\right)^{1/2}\log V \cdot \int\limits_{1/2+\varepsilon}^1 V^{\sigma}
\int\limits_{-\infty}^{\infty} \frac{|L(1/2+it, \chi)|}{(1+|t|)^{1+\varepsilon}} \sum\limits_{\rho\in {\mathcal N}(\psi,\sigma)} \frac{1}{(1+|t/c+\gamma|)^2}{\rm d} t{\rm d}\sigma.
\end{equation}
Employing H\"older's inequality, we deduce that the inner integral in (\ref{T0}) is bounded by
\begin{equation} \label{T01}
\ll \left(\int\limits_{-\infty}^{\infty} \frac{|L(1/2+it, \chi)|^4}{(1+|t|)^{1+\varepsilon}}
{\rm d} t\right)^{1/4} \left(\int\limits_{-\infty}^{\infty} \frac{1}{(1+|t|)^{1+\varepsilon}}\left(\sum\limits_{\rho\in {\mathcal N}(\psi,\sigma)} \frac{1}{(1+|t/c+\gamma|)^2}\right)^{4/3}{\rm d} t\right)^{3/4}.
\end{equation}
By Lemma \ref{zeronumb}, it follows that
\[ \sum\limits_{\rho\in {\mathcal N}(\psi,\sigma)} \frac{1}{(1+|t/c+\gamma|)^2} \ll
(l_2(1+|t|))^{\varepsilon}. \]
This implies that
\begin{equation} \label{T02}
\begin{split}
& \int\limits_{-\infty}^{\infty} \frac{1}{(1+|t|)^{1+\varepsilon}}\left(\sum\limits_{\rho\in {\mathcal N}(\psi,\sigma)} \frac{1}{(1+|t/c+\gamma|)^2}\right)^{4/3}{\rm d} t \ll l_2^{\varepsilon}\int\limits_{-\infty}^{\infty} \frac{1}{(1+|t|)^{1+\varepsilon/2}} \sum\limits_{\rho\in {\mathcal N}(\psi,\sigma)} \frac{1}{(1+|t/c+\gamma|)^2}{\rm d} t \\
&= l_2^{\varepsilon}\sum\limits_{\rho\in {\mathcal N}(\psi,\sigma)} \int\limits_{-\infty}^{\infty} \frac{{\rm d} t}{(1+|t|)^{1+\varepsilon/2}(1+|t/c+\gamma|)^2} \ll l_2^{\varepsilon}\sum\limits_{\rho\in {\mathcal N}(\psi,\sigma)} \frac{1}{(1+|\gamma|)^{1+\varepsilon/4}}.
\end{split}
\end{equation}
Combining (\ref{T0}), (\ref{T01}) and \eqref{T02}, and using partial summation and integration by parts, we obtain
\begin{equation} \label{T12}
\begin{split}
G_2\ll U^{1/2}\left(1+\frac{U^c}{V|q|}\right)^{1/2} & (l_2V)^{\varepsilon}
\int\limits_{1/2+\varepsilon}^{1} V^{\sigma} \left(
\int\limits_{0}^{\infty} \frac{1}{(1+|T|)^{2+\varepsilon}}
\int\limits_{-T}^{T} |L(1/2+it,\chi)|^{4} {\rm d} t {\rm d}
 T\right)^{1/4} \\
& \times \left(\int\limits_{0}^{\infty}\frac{N(\psi,T,\sigma)}{(1+|T|)^{2+\varepsilon}}{\rm d} T\right)^{3/4}{\rm d} \sigma.
\end{split}
\end{equation}

Let $\tilde Q(d,k,\chi)$ be the term $Q(d,k,\chi)$ with the summation over $p$ be extended to prime powers $p^m\sim P$. As remarked on the top of page 220 in \cite{Young}, trivial estimations show that the difference $|Q(d,k,\chi)-\tilde Q(d,k,\chi)|$ is negligible, thus we may replace $Q(d,k,\chi)$ by $\tilde Q(d,k,\chi)$. \newline

As shown on page 219 in \cite{Young}, $P^{3/2}(\log X)\tilde Q(d,k,\chi)$ is a term of the form \eqref{general}, where
\[
F(u,v)=\left(\frac{v}{P}\right)^{-3/2} g \left( \frac{ud_0}{H} \right) g \left( \frac{k}{K} \right) g \left( \frac{v}{P} \right) \hat{w}\left(\frac{ud_0A}{v},
\frac{kB}{v}\right)\hat{\phi}\left(\frac{\log v}{\log X}\right),
\]
the character $\chi$ in (\ref{general}) is replaced by $\overline{\chi}^3$ (with modulus
$l_1=k^2/d$), the character $\psi$ in (\ref{general}) is $\chi\psi_4(k/\cdot)$ (with modulus
$l_2=$lcm$(4,k^*,k^2/d)$, where $k^*$ is the conductor of $(k/\cdot)$), $q=-k^2/d_0^3$,
$U=\min\{H/d_0,P/(d_0A)\}$, $V=P$, and $c=3$. \newline

Now, by the above considerations, when $k\sim K$ and $\overline{\chi}^3$ is a non-trivial character, we obtain
\[
 \tilde Q(d,k,\chi)\ll Q_1(d,k,\chi)+Q_2(d,k,\chi)+Q_3(d,k,\chi),
\]
where
\[
 Q_1(d,k,\chi)=
P^{-1}\left(\frac{H}{d_0}\right)^{1/2}\left(1+\frac{H^{3/2}}{P^{1/2}K}\right)X^{\varepsilon}
\left(\int\limits_{0}^{\infty} \frac{1}{(1+T)^{2+\varepsilon}} \int\limits_{-T}^T |L(1/2+it,\overline{\chi}^3)|^2 {\rm d} t {\rm d} T
\right)^{1/2},
\]
\begin{equation*}
\begin{split}
 Q_2(d,k,\chi)
= P^{-3/2}
\left(\frac{H}{d_0}\right)^{1/2}\left(1+\frac{H^{3/2}}{P^{1/2}K}\right)X^{\varepsilon} \int\limits_{1/2+\varepsilon}^{1} P^{\sigma} & \left(
\int\limits_{0}^{\infty} \frac{1}{(1+|T|)^{2+\varepsilon}}
\int\limits_{-T}^{T} |L(1/2+it,\overline{\chi}^3)|^{4} {\rm d} t {\rm d}
T\right)^{1/4} \\
 & \times
 \left(\int\limits_{0}^{\infty}\frac{N(\chi\psi_4(k/\cdot),T,\sigma)}{(1+|T|)^{2+\varepsilon}}{\rm d} T\right)^{3/4}{\rm d} \sigma.
\end{split}
\end{equation*}
and $Q_3(d,k, \chi)=0$ if $\chi\psi_4(k/\cdot)$ is non-trivial and
\begin{equation} \label{Q3}
Q_3(d,k,\chi)= P^{-1/2}l^{3/16}\left(\frac{H}{d_0}\right)^{1/2}\left(1+\frac{H^{3/2}}{P^{1/2}K}\right)X^{\varepsilon}
\end{equation}
if $\chi\psi_4(k/\cdot)$ is trivial. Here $l$ denotes the conductor of $\overline{\chi}^3$.

\section{Splitting of $S(H,K,P)$ and estimation of some partial terms}
Let $l_1$ be the modulus of $\overline{\chi}^3$, $l_2$ be the modulus of $\psi=\chi\psi_4(k/\cdot)$, $\chi_0$ be the trivial character modulo $l_1$, and $\psi_0$ be the trivial character modulo $l_2$. \newline

As in \cite{Young}, we split $S(H,K,P)$ given by \eqref{transS1} into four terms $S_1$, $S_2$, $S_3$ and $S_4$,
where $S_1$ is the contribution of all characters $\chi$ such that
$\overline{\chi}^3\not=\chi_0$ and $\psi\not=\psi_0$, $S_2$ is the contribution of all $\chi$ such that $\overline{\chi}^3=\chi_0$ and $\psi\not=\psi_0$,
$S_3$ is the contribution of all $\chi$ such that
$\overline{\chi}^3=\chi_0$ and $\psi=\psi_0$, and $S_4$ is the contribution of all $\chi$ such that $\overline{\chi}^3\not=\chi_0$ and $\psi=\psi_0$. \newline

By the results of the previous section, we have
\begin{equation} \label{S11}
S_1\ll T_1+T_2+E,
\end{equation}
where
\[ T_i:=\sum\limits_{k\sim K} \sum\limits_{d|k^2} \frac{1}{\varphi(k^2/d)}
\sum\limits_{\chi\ \mbox{\rm \scriptsize mod } k^2/d} |\tau(\chi)| Q_i(d,k,\chi) \]
for $i=1,2$, and $E$ is a negligible error term coming from prime powers $p^m$ with
$m\ge 2$. \newline

Using the Cauchy-Schwarz inequality, Lemma \ref{second}, Lemma \ref{gauss} and Lemma \ref{charnumb}, we obtain
\[ T_1 \ll P^{-1}H^{1/2}\left(1+\frac{H^{3/2}}{P^{1/2}K}\right)X^{\varepsilon}\sum\limits_{k\sim K} k \sum\limits_{d|k^2} \frac{1}{(dd_0)^{1/2}} \ll P^{-1}H^{1/2}K^2\left(1+\frac{H^{3/2}}{P^{1/2}K}\right)X^{\varepsilon}. \]
Now a simple calculation using $H\ll (P/A)^{1+\varepsilon}$ and $K\ll (P/B)^{1+\varepsilon}$ shows that the desired estimate $T_1\ll X^{-\varepsilon}$ holds if $P\ll X^{7/9-\varepsilon}$. \newline

Using H\"older's inequality and again Lemma \ref{second}, Lemma \ref{gauss} and Lemma \ref{charnumb}, we obtain
\begin{equation} \label{TT2}
\begin{split}
T_2 \ll
 P^{-3/2}&H^{1/2}\left(1+\frac{H^{3/2}}{P^{1/2}K}\right)K^{-1/4}X^{\varepsilon} \\
& \times \int\limits_{1/2+\varepsilon}^{1} P^{\sigma} \left(\int\limits_{0}^{\infty}\frac{1}{(1+|T|)^{2+\varepsilon}}
\sum\limits_{k\sim K} \sum\limits_{d|k^2} \frac{d^{1/3}}{d_0^{2/3}}
 \sum\limits_{\chi \ \mbox{\rm \scriptsize mod } k^2/d}
 N(\chi\psi_4(k/\cdot),T,\sigma){\rm d} T\right)^{3/4}{\rm d}
 \sigma.
\end{split}
\end{equation}
We postpone the rather difficult estimation of this expression to section 11 and provide the required tools in sections 9 and 10. \newline

To estimate $S_4$, we consider two cases. When $K\ll P^{1/2}H^{-1/2}X^{-4\varepsilon}$,
we use the unconditional bound
\begin{equation} \label{MS4}
S_4\ll\frac{H^{1/2}K^{1/2}}{P^{1/2}}X^{\varepsilon}+\frac{H^{1/4}K^{1/2}}{P^{1/4}}
X^{\varepsilon}
\end{equation}
on page 222 in \cite{Young}. In this case, we find that the right-hand side of (\ref{MS4}) is
$\ll X^{-\varepsilon}$. When $K\gg P^{1/2}H^{-1/2}X^{-4\varepsilon}$, we estimate $S_4$ by
\[ S_4 \ll T_1+T_2+T_3+E, \]
where $T_1,T_2$ are defined as previously, $E$ is a negligible error term like in \eqref{S11}, and
\[ T_3:=\sum\limits_{k\sim K} \sum\limits_{d|k^2} \frac{1}{\varphi(k^2/d)}
\sum\limits_{\substack{\chi\ \mbox{\rm \scriptsize mod } k^2/d\\ \psi=\psi_0}} |\tau(\chi)| Q_3(d,k,\chi), \]
where $Q_3(d,k,\chi)$ is defined as in \eqref{Q3}. As shown above, $T_1\ll X^{-\varepsilon}$ holds if $P\ll X^{7/9-\varepsilon}$. The estimation of $T_2$ is postponed to section 11. For the term $T_3$, we find
$$
T_3\ll P^{-1/2}H^{1/2}X^{\varepsilon}\left(1+\frac{H^{3/2}}{P^{1/2}K}\right)\sum_{k\sim K}
\sum\limits_{d|k^2} \frac{1}{d_0^{1/2}\varphi(k^2/d)} \sum\limits_{\substack{\chi\ \mbox{\scriptsize\rm mod } k^2/d\\ \psi=\psi_0}} |\tau(\chi)|l^{3/16},
$$
where $l$ is the conductor of $\overline{\chi}^3$.
We apply Lemma \ref{gauss} to bound the Gauss sum $\tau(\chi)$ by $\ll k/\sqrt{d}$.
We further observe, as in \cite{Young}, that the character $\chi$ has conductor $k^*$ equal to the square-free part of $k$ (up to a factor 2 or 4) if $\chi\psi_4(k/\cdot)$ is trivial. This implies that $k^*|k^2d^{-1}$
and hence $d|k^2(k^*)^{-1}$. It also implies that the conductor of $\overline{\chi}^3$ does not exceed $4k$. Combining everything, we obtain
$$
T_3\ll P^{-1/2}H^{1/2}K^{3/16}X^{\varepsilon}\left(1+\frac{H^{3/2}}{P^{1/2}K}\right)\sum_{k\sim K}
\frac{1}{k} \sum\limits_{d|k^2(k^*)^{-1}} \frac{d^{1/2}}{d_0^{1/2}}.
$$
The expression on the right-hand side coincides, up to an extra factor of size $K^{3/16}$, with the bound for $S_4$ in \cite{Young}. By a similar computation as in \cite{Young}, this expression can be estimated by
\begin{equation} \label{T3est}
\ll P^{-1/2}H^{1/2}K^{3/16}X^{\varepsilon}\left(1+\frac{H^{3/2}}{P^{1/2}K}\right).
\end{equation}
We recall that we have assumed that $K\gg P^{1/2}H^{-1/2}X^{-4\varepsilon}$.
Using this and
$H\ll (P/A)^{1+\varepsilon}$, we find that the term in \eqref{T3est} and hence $T_3$ is bounded by $\ll X^{-\varepsilon}$ if $P\ll X^{77/96-\varepsilon}$, as desired.\newline

It remains to estimate $S_2+S_3$ and $T_2$ which is the content of the remainder of this paper.

\section{Estimation of $S_2+S_3$}
Now we bound
\[ S_2+S_3=\sum\limits_{k\sim K} \sum\limits_{d|k^2} \frac{1}{\phi(k^2/d)}
\sum\limits_{\substack{\chi\ \mbox{\scriptsize\rm mod } k^2/d\\ \overline{\chi}^3=\chi_0}}
\tau(\chi)\overline{\chi}(d_0^3/d)Q(d,k,\chi), \]
the contribution of characters $\chi$ such that $\overline{\chi}^3$ is trivial. Let $D$ be a real number with $1\le D\le (2K)^2$ which we specify later.
Let $S^{\sharp}(D)$ be the contribution of divisors $d\le D$ to $S_2+S_3$, and let $S^{\flat}(D)$ be the contribution of divisors $d>D$ to $S_2+S_3$. We first bound $S^{\sharp}(D)$. We rewrite this term in the form
\[ S^{\sharp}(D)=\sum\limits_{d\le D} \sum\limits_{k_1\sim K/d'} \frac{1}{\phi(k^2/d)}
\sum\limits_{\substack{\chi\ \mbox{\scriptsize\rm mod } k_1^2d^*\\ \overline{\chi}^3=\chi_0}}
\tau(\chi)\overline{\chi}(d_0^3/d)Q(d,k,\chi), \]
where $k=k_1d'$. Second, we remove the weight $U(h,d_0,k,p)$ by partial summation. Furthermore, we observe that by Lemma \ref{ord3}, the conductor of $\chi$ divides $3k_1d^*$. Hence, by Lemma \ref{gauss}, we can bound the Gauss sum by $\tau(\chi)\ll \sqrt{k_1d^*}$. Therefore the desired estimate $S^{\sharp}(D)\ll X^{-\varepsilon}$ holds if
\begin{equation} \label{aimS2}
P^{-3/2}K^{-2}\sum\limits_{d\le D} d\sqrt{d^*} \sum\limits_{k_1\sim K/d'}\
\sqrt{k_1} \sum\limits_{\substack{\chi\ \mbox{\scriptsize\rm mod } k_1^2d^*\\ \overline{\chi}^3=\chi_0}}\
\sum\limits_{p\sim P} \left\vert\sum\limits_{h\sim H/d_0} e\left(\frac{h^3d_0^3}{pk^2}\right)\chi_0(h)\right\vert \ll X^{-\varepsilon}
\end{equation}
holds. Now using Lemma \ref{charnumb} and \ref{expsum}, the left-hand side of \eqref{aimS2} is bounded by
\begin{eqnarray} \label{aimS3}
&\ll& P^{-3/2}K^{-2}X^{\varepsilon}\sum\limits_{d\le D} d\sqrt{d^*} \sum\limits_{k_1\sim K/d'} \sqrt{k_1}
\left(\left(\frac{H}{d_0}\right)^{1/2}P+
\left(\frac{H}{d_0}\right)^{1/4}P^{5/4}K^{1/2}d_0^{-3/4}\right)\\
&\ll& P^{-1/2}H^{1/2}K^{-1/2}X^{\varepsilon}\sum\limits_{d\le D}
\frac{d{d^*}^{1/2}}{d_0^{1/2}{d'}^{3/2}}+P^{-1/4}H^{1/4}X^{\varepsilon}
\sum\limits_{d\le D}
\frac{d{d^*}^{1/2}}{d_0{d'}^{3/2}}\nonumber\\
&\ll& P^{-1/2}H^{1/2}K^{-1/2}D^{1/2}X^{2\varepsilon}+
P^{-1/4}H^{1/4}D^{1/4}X^{2\varepsilon},\nonumber
\end{eqnarray}
where the third line arrives by the virtue of Lemma \ref{f}. We now set
\begin{equation} \label{D}
D:=\left\{ \begin{array}{llll} (2K)^2 & \mbox{ if } K \le P^{1/2}H^{-1/2}X^{-6\varepsilon},\\ \\
PH^{-1}X^{-12\varepsilon} & \mbox{ if } K \ge P^{1/2}H^{-1/2}X^{-6\varepsilon}.
\end{array} \right.
\end{equation}
Then from \eqref{aimS3}, we deduce that
\begin{equation} \label{sharpS}
S^{\sharp}(D)\ll X^{-\varepsilon}.
\end{equation}

Next we bound $S^{\flat}(D)$. If $K \le P^{1/2}H^{-1/2}X^{-6\varepsilon}$, then this expression is empty. Thus we may assume that
\begin{equation} \label{K}
K \ge P^{1/2}H^{-1/2}X^{-6\varepsilon}.
\end{equation}
We now write $k^2=de$ and use Lemma \ref{gauss} to obtain
\begin{equation}
S^{\flat}(D)\ll \sum\limits_{e\le (2K)^2/D} \frac{1}{\sqrt{e}}
\sum\limits_{\substack{\chi\ \mbox{\scriptsize\rm mod } e\\ \overline{\chi}^3=\chi_0}}\
\sum\limits_{\substack{k\sim K\\ e|k^2}} |Q(d,k,\chi)|.
\end{equation}
We remove the weight
\[ e\left(-\frac{h^3d_0^3}{pk^2}\right)U(h,d_0,k,p) \]
by partial summation, which leads to an extra factor of order of magnitude $P^{-3/2}\left(1+H^{3}/(PK^2)\right)$. Moreover, we estimate the character sum over $h$ trivially by
\begin{equation} \label{marge}
 \sum\limits_{h\sim H/d_0} \overline{\chi}^3(h) \ll \frac{H}{d_0}.
\end{equation}
Now our task is to prove that
\begin{equation} \label{ta}
P^{-3/2}H\left(1+\frac{H^{3}}{PK^2}\right)\sum\limits_{e\le (2K)^2/D} \frac{1}{\sqrt{e}}
\sum\limits_{\substack{\chi\ \mbox{\scriptsize\rm mod } e\\ \overline{\chi}^3=\chi_0}}\
\sum\limits_{\substack{k\sim K\\ e|k^2}} \frac{1}{d_0}
\left\vert\sum\limits_{p\sim P} \psi_4(p)\chi(p)\left(\frac{k}{p}\right)\right\vert\ll X^{-\varepsilon},
\end{equation}
where
\[ d_0=\left(k^2/e\right)_0 \]
which is the least integer $f$ such that $k^2/e$ is a divisor of $f^3$. \newline

We rewrite the inner double sum in \eqref{ta} in the form
\begin{equation} \label{tas}
\begin{split}
\sum\limits_{\substack{k\sim K\\ e|k^2}} \frac{1}{d_0} \left\vert\sum\limits_{p\sim P} \psi_4(p)\chi(p)\left(\frac{k}{p}\right)\right\vert & =
\sum\limits_{\substack{n\sim K/e'}} \frac{1}{(n^2e^*)_0} \left\vert\sum\limits_{p\sim P} \psi_4(p)\chi(p)\left(\frac{e'}{p}\right)\left(\frac{n}{p}\right)\right\vert,
\end{split}
\end{equation}
where $e^*$ is the square-free kernel of $e$, $e'$ is the least natural number such that
$e|{e'}^2$, and $k=ne'$. We further transform the right-hand side of \eqref{tas} as follows.
\begin{eqnarray} \label{tas1}
& & \sum\limits_{n\sim K/e'} \frac{1}{(n^2e^*)_0} \left\vert\sum\limits_{p\sim P} \psi_4(p)\chi(p)\left(\frac{e'}{p}\right)\left(\frac{n}{p}\right)\right\vert\\
&=& \sum\limits_{g|e^*} \sum\limits_{\substack{n\sim K/e'\\ (n,e^*)=g}}
\frac{1}{(n^2e^*)_0} \left\vert\sum\limits_{p\sim P} \psi_4(p)\chi(p)\left(\frac{e'}{p}\right)\left(\frac{n}{p}\right)\right\vert\nonumber\\
&=& \sum\limits_{g|e^*} \sum\limits_{\substack{m\sim K/(ge')\\ (m,e^*/g)=1}} \frac{1}{(m^2g^3(e^*/g))_0} \left\vert\sum\limits_{p\sim P} \psi_4(p)\chi(p)\left(\frac{ge'}{p}\right)\left(\frac{m}{p}\right)\right\vert\nonumber\\
&=& \frac{1}{e^*} \sum\limits_{g|e^*} \sum\limits_{\substack{m\sim K/(ge')\\ (m,e^*/g)=1}} \frac{1}{(m^2)_0} \left\vert\sum\limits_{p\sim P} \psi_4(p)\chi(p)\left(\frac{ge'}{p}\right)\left(\frac{m}{p}\right)\right\vert\nonumber\\
&\le& \frac{\sqrt{e'}}{e^*\sqrt{K}} \sum\limits_{g|e^*} \sqrt{g} \sum\limits_{m\le 2K/(ge')} \frac{\sqrt{m}}{(m^2)_0} \left\vert\sum\limits_{p\sim P} \psi_4(p)\chi(p)\left(\frac{ge'}{p}\right)\left(\frac{m}{p}\right)\right\vert,\nonumber
\end{eqnarray}
where we have used that $e^*$ is square-free and that the arithmetic function $f(x)=x_0$ is multiplicative.
The inner double sum in the last line of \eqref{tas1} can be re-written in the form
\begin{eqnarray} \label{inner}
& & \sum\limits_{m\le 2K/(ge')} \frac{\sqrt{m}}{(m^2)_0} \left\vert\sum\limits_{p\sim P} \psi_4(p)\chi(p)\left(\frac{ge'}{p}\right)\left(\frac{m}{p}\right)\right\vert\\
&=& \sum\limits_{p\sim P} \psi_4(p)\chi(p)\left(\frac{ge'}{p}\right) \sum\limits_{m\le 2K/(ge')} a_m \frac{\sqrt{m}}{(m^2)_0} \left(\frac{m}{p}\right),\nonumber
\end{eqnarray}
where $a_m$ are suitable complex numbers with $|a_m|=1$. Using the Cauchy-Schwarz inequality and Lemma~\ref{Heathls} due to Heath-Brown,
and taking into account that $K/(ge')\ll P$, we estimate the
right-hand side of \eqref{inner} by
\begin{equation} \label{afterheath}
\ll X^{\varepsilon}P \left(\sum\limits_{q\le 2K/(ge')} \ \sum\limits_{\substack{m_1,m_2\le 2K/(ge')\\
m_1m_2=q^2}} \frac{\sqrt{m_1}}{(m_1^2)_0}\frac{\sqrt{m_2}}{(m_2^2)_0} \right)^{1/2}.
\end{equation}
Using
$$
\frac{\sqrt{m_1}}{(m_1^2)_0}\frac{\sqrt{m_2}}{(m_2^2)_0}\le \frac{q}{(q^4)_0}\ \ \ \mbox{ if }\ \ \ m_1m_2=q^2
$$
and Lemma \ref{f}, \eqref{afterheath} is bounded by
$$
\ll X^{2\varepsilon}P.
$$
Using this and \eqref{inner} it follows that \eqref{tas1} and hence \eqref{tas} is bounded by
$$
X^{3\varepsilon}P\sqrt{\frac{e'}{e^*K}}.
$$
Using this, Lemma \ref{charnumb} and $\sqrt{ee^*}=e'$, we deduce that the left-hand side of \eqref{ta} is dominated by
$$
X^{4\varepsilon}P^{-1/2}HK^{-1/2}\left(1+\frac{H^{3}}{PK^2}\right)\sum\limits_{e\le (2K)^2/D} \frac{1}{\sqrt{e'}}.
$$
Now we again use Lemma \ref{f} to estimate the above by
\begin{equation} \label{alm}
X^{5\varepsilon}P^{-1/2}HK^{1/2}D^{-1/2}\left(1+\frac{H^{3}}{PK^2}\right).
\end{equation}

Remember that we assume that $K\ge P^{1/2}K^{-1/2}X^{-6\varepsilon}$ (see inequality \eqref{K}) and hence have $D=PH^{-1}X^{-12\varepsilon}$ by our definition of $D$ in \eqref{D}. Therefore,
\eqref{alm} and hence $S^{\flat}(D)$ is bounded by
\[ \ll X^{20\varepsilon}\left(P^{-1}K^{1/2}H^{3/2}+P^{-11/4}H^{21/4}\right). \]
Finally, using $H\ll (P/A)^{1+\varepsilon}$ and $K\ll (P/B)^{1+\varepsilon}$, we find that
\[ S^{\flat}(D) \ll X^{-\varepsilon} \]
if $P\ll X^{7/10-\varepsilon}$.  The appearance of this exponent marks the limit of our method. \newline
\newline

For the proof of Theorem~\ref{main} it now suffices to prove the bound $T_2\ll X^{-\varepsilon}$. This is the object of the remainder of this paper. \newline

\section{A general mean value estimate for Dirichlet polynomials}

For the estimation of $T_2$, we shall need zero density estimates for Dirichlet $L$-functions. In section 10, we shall establish such estimates in terms of mean values of Dirichlet polynomials. To bound them, we now prove the following general result.

\begin{theorem} \label{Dirichlet}
Let $M$ and $N$ be natural numbers, $(a_1,...,a_N)$ a vector in $\mathbbm{C}^N$ and
$(b_{m,n})$ an $M\times N$-matrix with complex entries. Set $a_n=0$ if $n>N$. Assume that for any $x\ge 0$ and $\Delta\ge 0$ an estimate of the form
\begin{equation} \label{direst}
\sum\limits_{m=1}^M \left\vert \sum\limits_{x<n\le x+\Delta} b_{m,n}a_n \right\vert^2 \le F\left(\Delta+G\right)\sum\limits_{x<n\le x+\Delta} |a_n|^2 \end{equation}
holds for some fixed $F,G>0$. Then for any $T\ge 1$, $\sigma\in \rear$ and any sets
$\mathcal{S}(m)$ of complex numbers $\rho$ such that $\sigma+1 \geq \beta = \Re \rho \geq
 \sigma$, $\gamma = \Im \rho$, $|\gamma|\le T$ and
\[
 \left| \Im \rho - \Im \rho' \right| \geq 1
\]
for all distinct $\rho$ and $\rho'$ in $\mathcal{S}(m)$, we have
\begin{equation} \label{diresteq}
\sum\limits_{m=1}^M \sum\limits_{\rho\in \mathcal{S}(m)} \left\vert \sum_{n=1}^N b_{m,n}a_nn^{-\rho} \right\vert^2 \ll (\log 2N)F\left(N+GT\right)
\sum\limits_{n=1}^N |a_n|^2n^{-2\sigma},
\end{equation}
where the implied $\ll$-constant is absolute.
\end{theorem}

\begin{proof} First, we re-write the inner-most sum on the left-hand
 side of \eqref{diresteq} as follows.
\[ \sum_{n=1}^N b_{m,n}a_nn^{-\rho} = a_1 b_{m,1} + \sum_{2 \leq n \leq  N} a_n b_{m,n} n^{-\sigma - i \gamma} n^{\sigma-\beta} = a_1 b_{m,1} + \int_2^N u^{\sigma - \beta} \dif S(\sigma+i \gamma,  m, u), \]
where
\[
 S(s,m,u) = \sum_{2 \leq n \leq u} a_n b_{m,n}n^{-s}.
\]
Applying integration by parts, we get that
\begin{equation} \label{diriest1}
\sum_{n=1}^N b_{m,n}a_nn^{-\rho} = a_1 b_{m,1} + S(\sigma+i\gamma,m,N) N^{\sigma-\beta} + (\beta - \sigma) \int_2^N S(\sigma+i\gamma,m,u) u^{-\beta+\sigma-1} \dif u.
\end{equation}
Now using Cauchy's inequality, we get from \eqref{diriest1} that
\begin{eqnarray*}
\left| \sum_{n=1}^N b_{m,n}a_nn^{-\rho} \right|^2 & \ll & |a_1
b_{m,1}|^2 + | S(\sigma+i \gamma, m, N) |^2 N^{2(\sigma-\beta)} + (\beta -\sigma)^2 \left( \int_2^N (\log u)
					u^{-2\beta+2\sigma-1} \dif
					u\right) \\
 && \hspace*{.5cm} \times \left( \int_2^N \left| S(\sigma+i \gamma, m, u)
								 \right|^2 \frac{\dif u}{u \log u} \right) \\
 & \ll &  |a_1 b_{m,1}|^2 + \left| S(\sigma+i \gamma, m, N)
			    \right|^2 N^{2(\sigma-\beta)} + \int_2^N
 \left| S(\sigma+i \gamma, m, u) \right|^2 \frac{\dif u}{u \log u}.
\end{eqnarray*}
We now note that
\begin{equation} \label{diriest2}
\sum_{\rho \in \mathcal{S}(m)} \left| S(\sigma+i\gamma,
			     m,u)\right|^2 = \sum_{\rho \in
\mathcal{S}(m)} \left| \sum_{2 \leq n \leq u} c_n b_{m,n} n^{-i \gamma} \right|^2,
\end{equation}
with $c_n=a_nn^{-\sigma}$.  Applying Lemma~\ref{gal1}, we get that the
 right-hand side of \eqref{diriest2} is
\begin{equation} \label{diriest3}
\begin{split}
\ll \int_{-T}^T &\left| \sum_{2 \leq n \leq u} c_n b_{m,n} n^{-it}
		     \right|^2 \dif t \\
& + \left( \int_{-T}^T \left| \sum_{2 \leq n \leq u} c_n b_{m,n} n^{-it}
		     \right|^2 \dif t \right)^{1/2} \left( \int_{-T}^T \left|
 \sum_{2 \leq n \leq u} c_n b_{m,n} n^{-it} \log n \right|^2 \dif t \right)^{1/2} ,
\end{split}
\end{equation}
recalling that the $\gamma$'s are well-spaced with spacing 1.  Now using
 Lemma~\ref{gal2}, we get that \eqref{diriest3} is
\begin{equation} \label{diriest4}
\begin{split}
\ll T^2 \int_0^N & \left| \sum_{y<n\leq \tau y} c_n b_{m,n} \right|^2
 \frac{\dif y}{y} + T^2 \left( \int_0^N \left| \sum_{y<n\leq \tau y} c_n
					b_{m,n} \right|^2 \frac{\dif
			 y}{y} \right)^{1/2} \\
 & \times \left( \int_0^N \left| \sum_{y<n\leq \tau y} c_n
					b_{m,n} \log n \right|^2
 \frac{\dif y}{y} \right)^{1/2}.
\end{split}
\end{equation}
where $\tau = \exp (1/T)$.  Now summing over $m$, \eqref{diriest2},
 and \eqref{diriest3} give that
\begin{equation} \label{diriest5}
\begin{split}
\sum_m& \sum_{\rho \in \mathcal{S}(m)} \left| S(\sigma+i\gamma, m,u)
				   \right|^2 \\
&\ll T^2 \int_0^N \sum_m \left| \sum_{y<n\leq \tau y} c_n
 b_{m,n}\right|^2 \frac{\dif y}{y} \\
& \hspace*{.5in} + T^2 \left( \int_0^N \sum_m \left| \sum_{y<n\leq \tau y} c_n
 b_{m,n}\right|^2 \frac{\dif y}{y} \right)^{1/2} \left( \int_0^N \sum_m \left| \sum_{y<n\leq \tau y} c_n
 b_{m,n} \log n \right|^2 \frac{\dif y}{y} \right)^{1/2},
\end{split}
\end{equation}
after applying Cauchy's inequality.  Now applying \eqref{direst} together with partial summation to the  right-hand side of \eqref{diriest5} after partial summation to remove the $n^{\sigma}$ and $n^{\sigma} \log n$, we get that \eqref{diriest5} is
\begin{eqnarray*}
& \ll & T^2 (\log 2N) \int_0^N F((\tau-1)y +G) \sum_{y<n\leq \tau y} |a_n|^2
 n^{-2\sigma} \frac{\dif y}{y} \\
& \ll & T^2 (\log 2N) F \sum_{n=1}^N |a_n|^2 n^{-2\sigma} \int_{n/\tau}^n
 ((\tau-1)y + G) \frac{\dif y}{y} \\
& \ll & (\log 2N) F \sum_{n=1}^N |a_n|^2 n^{-2\sigma}
 (T^2n(\tau-1)(1-1/\tau)+GT) \\
& \ll & (\log 2N) F (N+GT) \sum_{n=1}^N |a_n|^2 n^{-2\sigma}.
\end{eqnarray*}
The last of the above inequalities arrives by the following
 observation.  The elementary inequality
\[
 |e^x-1| \ll |x|, \; \mbox{for} \; |x| \leq 1
\]
implies that
\[
 |\tau -1| \ll |T|^{-1} \; \mbox{and} \; |1/\tau -1 | \ll |T|^{-1}, \;
 \mbox{for} \; |T| \geq 1.
\]
Now combining all estimates, we get the desired result.
\end{proof}

\section{A general zero density estimate}
In the sequel, let
\begin{equation}
{\rm d}\mu(v)=e^{-|v|}{\rm d}v+\delta(v),
\end{equation}
where ${\rm d}v$ is the Lebesgue measure on $\rear$, and $\delta(v)$ is the point measure at $v=0$. \newline

In this section, we establish the following general density estimate for zeros of Dirichlet $L$-functions in terms of mean values of Dirichlet polynomials twisted with characters.

\begin{theorem} \label{zerodens}
Let $T\ge 1$, $1/2+\varepsilon\le \sigma \le 1$ and $\mathcal{X}$ be a set of Dirichlet characters with maximum modulus $U$, and $c_\chi>0$ for $\chi\in \mathcal{X}$. Set $D:=UT$. By $\mathcal{R}$ denote a set of pairs $(\chi,\rho)$ such that $\chi\in \mathcal{X}$, $\Re \rho\ge \sigma$, $|\Im \rho|\le T$, and
the imaginary parts of the $\rho$'s belonging to a fixed character $\chi\in \mathcal{X}$ are well-spaced with spacing $\ge 1$. By $\mathcal{B}$ denote a sequence $(b_n^*(v))_{n\in\mathbbm{N}}$ of Lebesgue-integrable functions with domain $\rear$ and range $\mathbbm{C}$ which are bounded by
$|b_n^*(v)|\le 1$ for all $n\in \mathbbm{N}$, $v\in \rear$. Then
\begin{equation*}
\begin{split}
\sum\limits_{\chi\in \mathcal{X}} c_{\chi}N(\chi,T,\sigma)
\ll D^{\varepsilon}\ \sup\limits_{\mathcal{R }}\ \sup\limits_{\mathcal{B}}\ & \inf\limits_{ D^{1/2+\varepsilon}\le Z}\
\sup\limits_{Z\le P\le D^{1+\varepsilon}+Z^{3/2}} \ \sup\limits_{2\le \alpha \le C} \\
& \int\limits_{-\infty}^{\infty} \sum\limits_{(\chi,\rho)\in \mathcal{R}} c_{\chi}
\left\vert \sum\limits_{P<n\le \alpha P} b_n^*(v)\chi(n)n^{-\rho} \right\vert^2 {\rm d}\mu(v),
\end{split}
\end{equation*}
where $C\ge 2$ is a constant only depending on $\varepsilon$, and the implied $\ll$-constant depends only on $\varepsilon$ as well.
\end{theorem}

\begin{proof} We follow the method in section 10.4. in \cite{IwKo}. As in (10.74) in \cite{IwKo}, we choose
\begin{equation}
\mathcal{L}:=2\log D,\ \ \  \ X:=D^{1/2}\mathcal{L},\ \ \ \  Y:=D^{1/2}{\mathcal{L}}^2.
\end{equation}
We further choose
\begin{equation} \label{Mchoice}
M:=D^{\varepsilon/3}
\end{equation}
which consists with the conditions $1\le M\le D^{1/2}$, (10.80) and $M\ge D^{\varepsilon/4}$ on pages 261 and 262 in \cite{IwKo}. To see that our choice of $M$ consists with (10.80) in \cite{IwKo}, we note that our $\sigma$ was denoted by $\alpha$ in section 10.4. in \cite{IwKo}, and that in Theorem \ref{zerodens}  we assume $\sigma\ge 1/2+\varepsilon$.  \newline

Now, following the method on pages 260-262 in \cite{IwKo}, we find that
\[ N(\chi,T,\sigma)=\sum\limits_{l=1}^{L} R_l(\chi), \]
where $L=[\log Y/\log 2]$, and $R_l(\chi)$ is the cardinality of a certain subset $\mathcal{S}_l(\chi)$ of the zeros $\rho$ of $L(s,\chi)$ with $\Re \rho\ge \sigma$, $|\Im \rho|\le T$ which is bounded by
\begin{equation} \label{Rl}
R_l(\chi)\le \mathcal{L}^{6k} \int\limits_{-\infty}^{\infty} \sum\limits_{\rho\in \mathcal{S}_l(\chi)}
\left| \sum\limits_{P<n\le 2^{k}P} b_{n}(v)\chi(n)n^{-\rho}\right|^2 {\rm d}\mu(v)
\end{equation}
for all $l$.
In \eqref{Rl}, $k$ is a natural number that depends on $l$ but not on $\chi$ and is bounded from above by a constant which depends only on $\varepsilon$,
$P$ depends on $l$ but not on $\chi$ and falls into the segment
\[ Z\le P\le (MY)^2+Z^{3/2}, \]
where $Z$ is a fixed natural number that satisfies the inequality $MY\le Z$, and
$(b_n(v))_{n\in\mathbbm{N}}$ is a sequence, depending on $l$ but not on $\chi$, of certain Lebesgue-integrable functions with domain $\rear$ and range $\mathbbm{C}$ which are bounded by
$|b_n(v)|\le \tau_{4k}(n)$ (the divisor function of order $4k$) for all $n\in \mathbbm{N}$, $v\in \rear$.
Adding up \eqref{Rl} over $\chi$, and taking the above conditions to $k$, $P$ and
$b_n(v)$ into consideration, we deduce that for some constant $C\ge 2$ only depending on
$\varepsilon$, we have
\begin{equation} \label{lest}
\begin{split}
\sum\limits_{\chi\in \mathcal{X}} c_{\chi}R_l(\chi)
\ll D^{\varepsilon/2}\ \sup\limits_{\mathcal{R }}\ \sup\limits_{\mathcal{B}} \ &\inf\limits_{ D^{1/2+\varepsilon}\le Z }\
\sup\limits_{Z\le P\le D^{1+\varepsilon}+Z^{3/2}}\ \sup\limits_{2\le \alpha \le C}\\
& \int\limits_{-\infty}^{\infty}
\sum\limits_{(\chi,\rho)\in \mathcal{R}} c_{\chi}
\left\vert \sum\limits_{P<n\le \alpha P} b_n^*(v)\chi(n)n^{-\rho} \right\vert^2 {\rm d}\mu(v),
\end{split}
\end{equation}
where we have used that $\tau_{4k}(n)\ll n^{\varepsilon}$ and that the set $\mathcal{S}_l(\chi)$ can be divided into $O(\log D)$ subsets such that the imaginary parts of the elements of any of these subsets are well-spaced with spacing $\ge 1$.
Adding up \eqref{lest} over $l$, we obtain the desired result.
\end{proof}

Using Theorem \ref{zerodens}, we now derive the following more explicit zero density estimate.

\begin{theorem} \label{zerodensexp}
Let $T\ge 1$, $1/2+\varepsilon \le \sigma \le 1$, $\mathcal{X}$ be a set of Dirichlet characters with maximum modulus $U$, and $c_\chi>0$ for $\chi\in \mathcal{X}$. Set $D:=UT$.
Assume that for any $x\ge 0$, $\Delta\ge 0$ and any sequence $(a_n)_{n\in \mathbbm{N}}$ of complex numbers
an estimate of the form
\begin{equation}
\sum\limits_{\chi\in \mathcal{X}} c_{\chi}\left\vert \sum\limits_{x<n\le x+\Delta} a_n\chi(n) \right\vert^2 \le F\left(\Delta+G\right)\sum\limits_{x<n\le x+\Delta} |a_n|^2
\end{equation}
with
\begin{equation} \label{condi}
U^{3/4+2\varepsilon} \le G
\end{equation}
holds. Then we have
\begin{equation}
\sum\limits_{\chi\in \mathcal{X}} c_{\chi} N(\chi,T,\sigma) \ll D^{3\varepsilon}
F\left((GT)^{\frac{3(1-\sigma)}{2-\sigma}}+
D^{2(1-\sigma)}\right),\nonumber
\end{equation}
where the implied $\ll$-constant depends only on $\varepsilon$.
\end{theorem}

\begin{proof} Using Theorem \ref{zerodens} and Theorem \ref{Dirichlet} with $b_{m,n}=\sqrt{c_{\chi}}\chi(n)$, we obtain
\begin{equation} \label{zero}
\begin{split}
 \sum\limits_{\chi\in \mathcal{X}} c_{\chi}N(\chi,T,\sigma) & \ll D^{2\varepsilon}
F\inf\limits_{D^{1/2+\varepsilon}\le Z}
\sup\limits_{Z\le P\le D^{1+\varepsilon}+Z^{3/2}} \left( P+GT\right)
P^{1-2\sigma} \\
&\ll D^{3\varepsilon} F\left(D^{2(1-\sigma)}+\inf\limits_{D^{1/2+\varepsilon} \le Z} \left(Z^{3(1-\sigma)}+ GTZ^{1-2\sigma}\right)\right).
\end{split}
\end{equation}
We choose
\begin{equation}
Z:=\left(GT\right)^{1/(2-\sigma)}.
\end{equation}
By \eqref{condi}, this choice doesn't contradict our requirement that $D^{1/2+\varepsilon}\le Z$. Now the desired result follows from \eqref{zero}.
\end{proof}

\section{Estimation of $T_2$}

We rewrite the inner triple sum in \eqref{TT2} in the form
\begin{equation} \label{rew}
\begin{split}
\sum\limits_{k\sim K} \sum\limits_{d|k^2} \frac{d^{1/3}}{d_0^{2/3}} \sum\limits_{\chi \ \mbox{\rm \scriptsize mod } k^2/d} N(\chi\psi_4(k/\cdot),T,\sigma) &= \sum\limits_{d\le (2K)^2} \frac{d^{1/3}}{d_0^{2/3}}
\sum\limits_{k_1\sim K/d'} \sum\limits_{\chi \ \mbox{\rm \scriptsize mod } k_1^2d^*}  N(\chi\psi_4(k_1d'/\cdot),T,\sigma) \\
&= \sum\limits_{d\le (2K)^2} \frac{d^{1/3}}{d_0^{2/3}} \sum\limits_{k_1\sim K/d'} \sum\limits_{\chi \ \mbox{\rm \scriptsize mod } k_1^2d^*}  N(\psi_4(d'/\cdot)\chi,T,\sigma),
\end{split}
\end{equation}
where we have used that $\chi(k_1/\cdot)$ runs over all characters modulo $k_1^2d^*$ as $\chi$ runs over all characters modulo $k_1^2d^*$.
For the estimation of the right-hand side of \eqref{rew} we shall use Theorem \ref{zerodensexp}. To this end, we need to bound the character sum
\begin{equation} \label{post1}
 \sum\limits_{d\le (2K)^2} \frac{d^{1/3}}{d_0^{2/3}}
\sum\limits_{k_1\sim K/d'}\ \sum\limits_{\chi \ \mbox{\rm \scriptsize mod } k_1^2d^*} \left\vert \sum\limits_{x<n\le x+\Delta} a_n\psi_4(n)\left(\frac{d'}{n}\right)\chi(n)\right\vert^2.
\end{equation}
Using Lemma \ref{ls}, \eqref{post1} is
\begin{equation} \label{t2est}
\ll X^{\varepsilon}
\left(K\Delta\sum\limits_{d\le (2K)^2} \frac{d^{1/3}}{d_0^{2/3}d'}+K^3\sum\limits_{d\le (2K)^2} \frac{d^{1/3}{d^*}}{d_0^{2/3}{d'}^3}\right)\sum\limits_{x<n\le x+\Delta} |a_n|^2.
\end{equation}
By the virtue of Lemma \ref{f}, we deduce that \eqref{t2est} is bounded by
\begin{equation} \label{t2esti}
\ll X^{\varepsilon}K\left(\Delta+K^{2}\right)\sum\limits_{x<n\le x+\Delta} |a_n|^2.
\end{equation}

Let
\[ M:=P^{-3/2}H^{1/2}\left(1+\frac{H^{3/2}}{P^{1/2}K}\right). \]
Now using \eqref{TT2}, \eqref{t2esti},
Theorem \ref{zerodensexp} and the inequality
\[ 2(1-\sigma)\le \frac{3(1-\sigma)}{2-\sigma}, \]
valid for $1/2\le \sigma\le 1$,
we obtain
\begin{equation} \label{densapply}
T_2\ll MX^{\varepsilon} K^{1/2}\int\limits_{1/2}^{1}
P^{\sigma}K^{a(\sigma)}{\rm d}\sigma,
\end{equation}
where
\[ a(\sigma)=\frac{9(1-\sigma)}{2(2-\sigma)}. \]
We split the right-hand side of \eqref{densapply} into
\[ MX^{\varepsilon}K^{1/2}
\left(\int\limits_{1/2}^{7/8}+\int\limits_{7/8}^1\right)P^{\sigma}
K^{a(\sigma)}{\rm d}\sigma=:T_2^{\sharp}+T_2^{\flat}, \ \ \ \mbox{ say.} \]

First we deal with the term $T_2^{\sharp}$. We note that the integral $\int_{1/2}^{7/8}$ can be replaced by a supremum $\sup_{1/2\le \sigma\le 7/8}$. We desire to have the bound
\begin{equation} \label{des}
T_2^{\sharp}\ll X^{-\varepsilon}.
\end{equation}
Observing that $1/2+a(\sigma)-1\ge 0$ if $1/2\le \sigma\le 7/8$, and using that $H\ll (P/A)^{1+\varepsilon}$ and $K\ll (P/B)^{1+\varepsilon}$, we obtain after a short calculation that $\eqref{des}$ holds if
\[ P\ll X^{F(\sigma)-\varepsilon} \]
for all $1/2\le \sigma\le 7/8$, where
\begin{equation} \label{req}
F(\sigma):=\frac{37-32\sigma}{6(7-4\sigma-2\sigma^2)}.
\end{equation}
The minimum of $F(\sigma)$ in the range $1/2\le \sigma\le 7/8$ is attained at
\[ \sigma_1=\frac{37-\sqrt{153}}{32}=0.7697... \]
and is $F(\sigma_1)=0.7534...$. Hence, our requirement for \eqref{des} to hold is
\begin{equation} \label{finalcond}
P\ll X^{F(\sigma_1)-\varepsilon}=X^{0.7534...}.
\end{equation}

To bound $T_2^{\flat}$, we observe that
\begin{equation} \label{bart}
T_2^{\flat} \ll MX^{\varepsilon}K^{1/2}\left( P^{7/8}K^{9(1-7/8)/2} + P \right).
\end{equation}
By a short calculation using \eqref{small}, $H\ll (P/A)^{1+\varepsilon}$ and $K\ll (P/B)^{1+\varepsilon}$,
we see that the right-hand side of \eqref{bart} is $\ll X^{-\varepsilon}$ if
$P\ll X^{19/27-\varepsilon}$, as desired. This completes the proof of Theorem \ref{main}.\\

\section{Notes}
An improvement of our exponent $7/10$ is conceivable if we use a large sieve inequality for sextic characters, recently established in \cite{BaYo}, to estimate a part of the term $S_2+S_3$ with $d$'s in a certain range. We may be able to improve further if we use an alternative method for bounding the part of $S_1$ with large $d$ and then large sieve with square moduli, developed both jointly and independently by the authors in \cites{SBLZ3, BaZh, SB1, Zha}, for bounding the part of $S_1$ with small $d$. However, this will become quite complicated. \newline

Moreover, it would be highly interesting to have a completely unconditional majorant for the average rank of all elliptic curves, i.e. an unconditional version of Corollary~\ref{cor1}.  Recall that Corollary~\ref{cor1} uses GRH for Hasse-Weil $L$-functions to discard all zeros that are not central by positivity of $\phi$ on the real line.  Therefore, fore-going the said GRH, it should be possible to use zero-density estimates to handle the zeros far off the critical line unconditionally.  This has been done by E. Kowalski, P. Michel and J. VanderKam in \cites{KM1, KM2, KMVdK} for $L$-functions associated with weight 2 level $q$ new forms, {\it \'a la} A. Selberg \cite{sel} whose method uses the approximate functional equations of the relevant $L$-functions.  Additional knowledge of the often problematic root numbers appearing in these approximate functional equations enabled the authors of \cites{KM1, KM2, KMVdK} to obtain their result.  Unfortunately, having such a zero-density result seems to be out of reach for the family of $L$-functions under consideration in this paper.  The methods in \cites{KM1, KM2, KMVdK} would not carry over for the family of all elliptic curve $L$-functions as the root numbers in this family are less well-known and present a much greater challenge.  Writing out this root number explicitly, we obtain an expression that contains the term
\[ \mu(4a^3+27b^2), \]
$\mu(n)$ being the M\"obius $\mu$ function which is usually extremely difficult to handle. \newline

\noindent{\bf Acknowledgments.} This work was completed when S. B. was visiting the Division of Mathematical Sciences of Nanyang Technological University(NTU) in Singapore.  He would like to thank NTU for its generous financial support and warm hospitality. S. B. further wishes to thank Jacobs University Bremen for providing excellent working conditions. L. Z. was supported by a postdoctoral fellowship at the {\it Institutionen f\"ur Matematik} of {\it Kungliga Tekniska H\"ogskolan}(KTH) in Stockholm and by a grant from the G\"oran Gustafsson Foundation during this work.  He wishes to thank these sources.   More in particular, for the help and advise given to him during his stay at KTH, L. Z. owes a debt of gratitude to Prof. P\"ar M. Kurlberg. Furthermore, both authors are grateful to Matthew Young for his helpful discussions with the authors on this problem. Finally, both authors wish to thank the referee for his/her valuable comments and pointing out an error in an earlier version of the manuscript.

\bibliography{biblio}
\bibliographystyle{amsxport}

\vspace*{.5cm}

\noindent School of Engineering \& Science, Jacobs Univ. Bremen \newline
P. O. Box 750561, Bremen 28725 Germany \newline
Email: {\tt s.baier@jacobs-university.de} \newline

\noindent Division of Math. Sci., School of Phys. \& Math. Sci., \newline
Nanyang Technological Univ., 637371 Singapore \newline
Email: {\tt lzhao@pmail.ntu.edu.sg}

\end{document}